\documentclass[10pt, reqno]{amsart}

\usepackage[T1]{fontenc}
\usepackage[english]{babel}
\usepackage[dvipsnames,table,xcdraw]{xcolor}
\definecolor{myred}{RGB}{228,26,28}
\definecolor{myblue}{RGB}{55,124,184}
\definecolor{mygreen}{RGB}{77,175,74}
\usepackage[colorlinks=true,urlcolor=myblue,linkcolor=myblue,runcolor=myred,citecolor=myred,linktoc=all]{hyperref}

\usepackage{geometry}
\usepackage{graphicx}
\usepackage{booktabs}
\usepackage{array}
\usepackage{wrapfig}
\usepackage{enumitem}
\setlist[itemize]{label={$\vcenter{\hbox{\tiny$\bullet$}}$}}
\usepackage{appendix}
\usepackage{caption}
\usepackage{subcaption}
\usepackage{csquotes}
\usepackage{todonotes}

\usepackage{amsmath,mathtools,amssymb,amsthm,amsfonts,esint}
\usepackage[foot]{amsaddr}
\usepackage{bm,chemformula}
\usepackage[boxruled, lined]{algorithm2e}
\let\originalleft\left
\let\originalright\right
\renewcommand{\left}{\mathopen{}\mathclose\bgroup\originalleft}
\renewcommand{\right}{\aftergroup\egroup\originalright}

\def\C{{\mathbb C}}
\def\R{{\mathbb R}}
\def\N{{\mathbb N}}
\def\Z{{\mathbb Z}}

\def\Hc{{\mathcal H}}
\def\Rc{{\mathcal R}}

\def\Nel{{N_{\rm el}}}
\def\Nex{{N_{\rm ex}}}
\def\Nb{{N_{\rm b}}}
\def\Np{{N_{\rm p}}}
\def\L{{\rm L}}
\def\H{{\rm H}}
\def\i{{\rm i}}
\newcommand{\eg}{\emph{e.g.}}
\renewcommand{\Re}{{\rm Re}}
\newcommand{\Ecut}{E_{\rm cut}}
\newcommand{\abs}[1]{\left| #1 \right|}

\newcommand{\Span}{{\rm Span}}

\newcommand{\Ran}{{\rm Ran}}
\newcommand{\Forall}{\forall\ }

\newcommand{\prt}[1]{\left( #1 \right)}
\newcommand{\set}[1]{\left\{ #1 \right\}}
\newcommand{\cro}[1]{\left\langle #1 \right\rangle}
\newcommand{\norm}[1]{\left\| #1 \right\|}
\newcommand{\eF}{\varepsilon_{\rm F}}
\newcommand{\e}{{\rm e}}
\renewcommand{\d}{{\rm d}}
\newcommand{\dH}{\delta V}

\newcommand{\df}{\delta f}
\newcommand{\drho}{\delta \rho}
\newcommand{\dphi}{\delta \phi}
\newcommand{\deF}{\delta\eF}
\newcommand{\ket}[1]{\left|#1\right\rangle}
\newcommand{\bra}[1]{\left\langle#1\right|}

\newcommand{\wt}{\widetilde}
\newcommand{\wb}{\overline}
\renewcommand{\lq}{\leqslant}
\newcommand{\gq}{\geqslant}

\theoremstyle{definition}
\newtheorem{remark}{Remark}

\newtheorem{proposition}{Proposition}

\addto\extrasenglish{%

}

\title[Numerical stability and efficiency of response property calculations in
DFT]{Numerical stability and efficiency of response property calculations in
  density functional theory}
\author{Eric Canc\`es$^{1,2}$}
\author{Michael F. Herbst$^{3}$}
\author{Gaspard Kemlin$^{1,2}$}
\author{Antoine Levitt$^{1,2}$}
\author{Benjamin Stamm$^{4}$}
\address[1]{CERMICS, ENPC, France}
\address[2]{MATHERIALS team, Inria Paris, France}
\address[3]{Applied and Computational Mathematics, RWTH Aachen University, Germany}
\address[4]{IANS, University of Stuttgart, Germany}
\email{\tiny eric.cances@enpc.fr, herbst@acom.rwth-aachen.de, gaspard.kemlin@enpc.fr,
  antoine.levitt@inria.fr, best@ians.uni-stuttgart.de}

\begin{document}

\begin{abstract}
  Response calculations in density functional theory aim at computing
  the change in ground-state density induced by an external perturbation.
  At finite temperature these are usually performed by computing variations of
  orbitals, which involve the iterative solution of potentially
  badly-conditioned linear systems, the Sternheimer equations.
  Since many sets of variations of
  orbitals yield the same variation of density matrix
  this involves a choice of gauge.
  Taking a numerical analysis point of view
  we present the various gauge choices proposed in the literature
  in a common framework and study their stability.
  Beyond existing methods we propose a new approach,
  based on a Schur complement using extra orbitals
  from the self-consistent-field calculations,
  to improve the stability and efficiency of the iterative solution of
  Sternheimer equations. We show the success of this
  strategy on nontrivial examples of practical interest,
  such as Heusler transition metal alloy compounds,
  where savings of around 40\% in the number of required
  cost-determining Hamiltonian applications have been achieved.
\end{abstract}

\maketitle

\section{Introduction}

Kohn-Sham (KS) density-functional theory (DFT) \cite{hohenbergInhomogeneousElectronGas1964,
  kohnSelfconsistentEquationsIncluding1965} is the most popular
approximation to the electronic many-body problem in quantum chemistry and materials science. While not perfect, it offers a favourable
compromise between accuracy and computational efficiency for a vast majority of
molecular systems and materials. In this work, we focus on KS-DFT approaches aiming at
computing electronic ground-state (GS) properties.
Having solved the minimization problem underlying DFT directly yields
the ground-state density and corresponding energy.
However, many quantities of interest, such as interatomic forces,
(hyper)polarizabilities, magnetic susceptibilities, phonons spectra, or
transport coefficients, correspond physically to the response of GS quantities
to a change in external parameters (e.g.~nuclear positions, electromagnetic fields).
As such their mathematical
expressions involve derivatives of the obtained GS solution
with respect to these parameters.
For example interatomic forces are {\em first-order} derivatives of the GS energy with
respect to the atomic positions, and can actually be obtained without computing
the derivatives of the GS density, thanks to the Hellmann-Feynman theorem
\cite{hellmannEinfuehrungQuantenchemie1944}. On the other hand the computation
of any property corresponding to {\em second- or higher-order} derivatives of
the GS energy does require the computation of derivatives of the density. More
precisely, it follows from Wigner's $(2n+1)$ theorem that $n^{\rm th}$-order
derivatives of the GS density are required to compute properties corresponding to
$(2n)^{\rm th}$- and $(2n+1)^{\rm st}$-derivatives of the KS energy functional.
More recent applications, such as the design of machine-learned
exchange-correlation energy functionals, also require the computation of
derivatives of the ground-state with respect to
parameters, such as the ones defining the exchange-correlation functional
\cite{kasimLearningExchangeCorrelationFunctional2021,
  kirkpatrickPushingFrontiersDensity2021,liKohnShamEquationsRegularizer2021}.

\medskip

Efficient numerical methods for evaluating these derivatives are therefore needed.
The application of generic perturbation theory to the special case of DFT is
known as density-functional perturbation theory (DFPT)
\cite{baroniGreenSfunctionApproach1987,
  gonzeAdiabaticDensityfunctionalPerturbation1995,
  gonzePerturbationExpansionVariational1995,
  gonzeDensityfunctionalApproachNonlinearresponse1989}. See also
\cite{normanPrinciplesPracticesMolecular2018} for applications to quantum
chemistry, \cite{baroniPhononsRelatedCrystal2001} for applications to phonon
calculations, and
\cite{cancesMathematicalPerspectiveDensity2014} for a mathematical analysis of
DFPT within the reduced Hartree-Fock (rHF) approximation (also called the Hartree approximation in the physics literature).
Although the practical implementation of first- and higher-order derivatives
computed by DFPT in electronic structure calculation software can be greatly
simplified by Automatic Differentiation techniques
\cite{griewankEvaluatingDerivativesPrinciples2008}, the efficiency of the
resulting code crucially depends on the efficiency of a key building block: the
computation of the linear response $\delta\rho$ of the GS density to an
infinitesimal variation $\delta V$ of the total Kohn-Sham potential.

\medskip

For reasons that will be detailed below, the numerical evaluation of the linear
map $\delta V \mapsto \delta \rho$ is not straightforward, especially
for periodic
metallic systems. Indeed, DFT calculations for metallic systems usually require
the introduction of a smearing temperature $T$, a numerical parameter which has
nothing to do with the physical temperature (in practice, its value is often higher than the
melting temperature of the crystal). For the sake of simplicity, let us first
consider the case of a periodic simulation cell $\Omega$ containing an even
number $N_{\rm el}$ of electrons in a spin-unpolarized state (see
\autoref{rmk:Brillouin} for details on how this formalism allows
for the computation of
properties of perfect crystals).
The Kohn-Sham GS at finite temperature $T>0$ is then described by an
$\L^2(\Omega)$-orthonormal set of orbitals $(\phi_n)_{n\in\N^*}$ with
energies $(\varepsilon_n)_{n\in\N^*}$, which are the eigenmodes of the Kohn-Sham
Hamiltonian $H$ associated with the GS density:
$$
H \phi_n = \varepsilon_n \phi_n, \qquad \int_\Omega \phi_m^*(\bm r) \phi_n(\bm
r) \d\bm r=\delta_{mn}, \qquad  \varepsilon_1 \lq \varepsilon_2 \lq
\varepsilon_3 \lq \cdots,
$$
together with periodic boundary conditions.
The GS density in turn reads
\begin{equation}\label{eq:rho_intro}
  \rho(\bm r) = \sum_{n=1}^{+\infty} f_n\abs{\phi_n(\bm r)}^2
  \quad\text{with}\quad f_n \coloneqq f\prt{\frac{\varepsilon_n-\eF}T},
\end{equation}
where $f$ is a smooth occupation function converging to $2$ at $-\infty$ and to
zero at $+\infty$, \eg~the Fermi-Dirac smearing function
$f(x)=\frac{2}{1+\e^x}$ (see \autoref{fig:Fermi-Dirac}). The Fermi level $\eF$ is the Lagrange multiplier of the
neutrality charge constraint: it is the unique real number such that
$$
\int_\Omega \rho(\bm r) \d\bm r = \sum_{n=1}^{+\infty} f_n =
\sum_{n=1}^{+\infty} f\prt{\frac{\varepsilon_n-\eF}T}=N_{\rm el}.
$$
It follows from perturbation theory that the linear response $\drho$ of the
density to an infinitesimal variation $\delta V$ of the {\em total} Kohn-Sham
potential is given by
$$
\delta\rho = \chi_0 \delta V,
$$
where $\chi_0$ is the independent-particle susceptibility operator (also called noninteracting
density response function). Equivalently, this operator describes the linear response of a system
of {\em noninteracting} electrons of density $\rho$ subject to an infinitesimal perturbation $\dH$.
It holds (see \autoref{sec:response})
\begin{equation}\label{eq:reponse_intro}
  \drho(\bm r) \coloneqq
  (\chi_0\dH)(\bm{r}) = \sum_{n=1}^{+\infty} \sum_{m=1}^{+\infty}
  \frac{f_{n}-f_{m}}{\varepsilon_{n}-\varepsilon_{m}} \phi_n^*(\bm{r})\phi_m(\bm{r})
  \prt{\dH_{mn} - \deF\delta_{mn}},
\end{equation}
where $\dH_{mn} \coloneqq \cro{\phi_m,\dH\phi_n}$, $\deF$ is the induced variation of
the Fermi level $\eF$, and $\delta_{mn}$ is the Kronecker symbol. We also use
the convention
$(f_n-f_n)/(\varepsilon_n-\varepsilon_n) =
\frac1Tf'\prt{\frac{\varepsilon_n-\eF}T}$.

\medskip

In practice, these equations are discretized on a finite basis set, so
that the sums in \eqref{eq:rho_intro} and \eqref{eq:reponse_intro}
become finite. Since the number of basis functions
$N_{\rm b}$ is often very large compared to the number of electrons in
the system, it is very expensive to compute the sums as such. However,
in practice it is
possible to restrict to the computation of a number $N \ll N_{\rm b}$
of orbitals. These orbitals are then computed using efficient iterative methods
\cite{payne1992iterative}.

\begin{figure}
  \centering
  \medskip
  \hfill\includegraphics[width=0.33\linewidth]{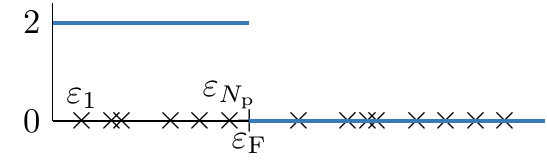}\hfill
  \includegraphics[width=0.33\linewidth]{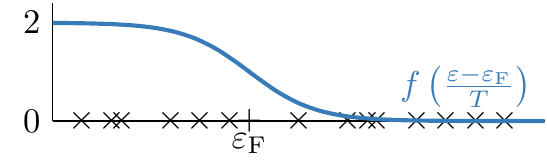}\hfill{~}
  \caption{The occupation numbers $f_n$ for $T=0$ (left) and $T>0$ (right).}
  \label{fig:Fermi-Dirac}
\end{figure}

\medskip

For insulating systems, there is a (possibly) large band gap between
$\varepsilon_{N_{\rm p}}$ and $\varepsilon_{N_{\rm p}+1}$ which
remains non-zero in the thermodynamic limit of a growing simulation
cell. As a result, the calculation can be done at zero temperature,
such that the occupation function $f$ becomes a step function (see
\autoref{fig:Fermi-Dirac}). The jump from $2$ to $0$ in the occupations
occurs exactly when the lowest $N_{\rm p}=N_{\rm el}/2$ energy levels
$\varepsilon_1 \lq \cdots \lq \varepsilon_{N_{\rm p}}$ are occupied
with an electron pair (two electrons of opposite spin). Thus, $f_n=2$
for $1 \lq n \lq N_{\rm p}$ and $f_n=0$ for $n > N_{\rm p}$. As a
result, $N$ can be chosen equal to the number of electron pairs
$N_{\rm p}$ without any approximation. In contrast, for metallic
systems $\varepsilon_{N_{\rm p}} = \varepsilon_{N_{\rm p}+1} = \eF$ in
the zero-temperature thermodynamic limit (more precisely there is a positive density
of states at the Fermi level in the limit of an infinite simulation
cell), causing the denominators in the right-hand side of formula \eqref{eq:reponse_intro}
to formally blow up. Calculations on metallic systems are thus done at
finite temperature $T>0$, in which case every orbital has a fractional
occupancy $f_n\in(0,2)$. However, since from a classical semiclassical approximation $\varepsilon_n$ tends to
infinity as $n^{2/3}$ as $n \to \infty$, and $f$ decays very quickly, one can safely assume that
only a finite number $N$ of orbitals have nonnegligible occupancies.
This allows one to avoid computing $\phi_n$ for $n>N$.
Under this approximation, a formal differentiation of
\eqref{eq:rho_intro} gives
\begin{equation}\label{eq:drho_intro}
  \drho(\bm r) = \sum_{n=1}^N f_n(\phi_n^*(\bm r)\dphi_n(\bm r) + \dphi_n^*(\bm
  r)\phi_n(\bm r)) + \df_n\abs{\phi_n(\bm r)}^2.
\end{equation}
However, while the response $\drho$ to a given $\delta V$ is well-defined by \eqref{eq:reponse_intro}, the
set $(\dphi_n,\df_n)_{1\lq n\lq N}$ is not. Indeed, the Kohn-Sham energy
functional being in fact a function of the density matrix $\gamma =
\sum_{n=1}^{N} f_n \ket{\phi_n}\bra{\phi_n}$, any transformation
of $(\dphi_n,\df_n)_{1\lq n\lq N}$ leaving invariant the first-order variation
\begin{equation}\label{eq:varDM}
  \delta\gamma \coloneqq \sum_{n=1}^{N} \delta f_n |\phi_n\rangle \langle \phi_n|+
  \sum_{n=1}^{N} f_n \left(  |\phi_n\rangle \langle \dphi_n|+ |\dphi_n\rangle
    \langle \phi_n| \right)
\end{equation}
of the density matrix is admissible. This gauge freedom can be used to
stabilize linear response calculations or, in the contrary, may lead to
numerical instabilities. Denote by $P$ the orthogonal projector onto
$\Span(\phi_n)_{1\lq n\lq N}$, the space spanned by the orbitals considered as
(partially) occupied, and by
$Q=1-P$ the orthogonal projector onto the space $\Span(\phi_n)_{n > N}$ spanned
by the orbitals considered as unoccupied. Then, the linear response of any
occupied orbital can be decomposed as $\dphi_n=\dphi_n^P +\dphi_n^Q$ where:
\begin{itemize}
  \item $\dphi_n^P = P\dphi_n\in\Ran(P)$ can be directly computed \emph{via} a
    sum-over-state formula (explicit decomposition on the basis of
    $(\phi_n)_{n \lq N}$). This contribution can be chosen to vanish in the
    zero-temperature limit, as in that case
    $P\, \delta\gamma\, P = 0$. At finite temperature,
    a gauge choice has to be made and several options have been
    proposed in the literature;
  \item $\dphi_n^Q = Q\dphi_n\in\Ran(Q)$ is the unique solution of the
    so-called Sternheimer equation
    \cite{sternheimerElectronicPolarizabilitiesIons1954}
    \begin{equation}\label{eq:Q_intro}
      Q(H-\varepsilon_n)Q\dphi_n^Q = -Q\dH\phi_n,
    \end{equation}
    where $H$ is the Kohn-Sham Hamiltonian of the system.
    This equation is possibly very ill-conditioned for $n=N$ if $\varepsilon_{N+1} -
    \varepsilon_N$ is very small.
\end{itemize}

\medskip

This paper addresses these two issues. First, we review and analyse the different gauge choices for
$\dphi_n^P$ proposed in the literature and introduce a new one. We
bring all these various gauge choices together in a new common framework and
analyse their performance in terms of numerical stability.
Second, for the contribution $\dphi_n^Q$, we investigate how to improve the
conditioning of the linear system $\eqref{eq:Q_intro}$, which is usually solved
with iterative solvers and we propose a new approach.
This new approach is based on the fact that, as a byproduct of the
iterative computation of the ground state orbitals
$(\phi_{n})_{n \lq N}$, one usually obtains relatively good
approximations of the following eigenvectors. This information is
often discarded for response calculations; we use them in a Schur
complement approach to improve the conditioning of the iterative solve
of the Sternheimer equation. We quantify the improvement of the
conditioning obtained by this new approach and illustrate its efficiency on
several metallic systems, from aluminium to transition metal alloys. We observe
a reduction of typically 40\% of the number of Hamiltonian applications (the
most costly step of the calculation). The numerical tests have been performed
with the DFTK
software \cite{herbstDFTKJulianApproach2021}, a recently developed plane-wave DFT
package in
\texttt{Julia}  allowing for both easy implementation of novel algorithms and
numerical simulations of challenging systems of physical interest.
The improvements suggested in this work are now the default choice in DFTK
to solve response problems.

\medskip

This paper is organized as follows. In \autoref{sec:math}, we review the
periodic KS-DFT equations and the associated approximations.  We also
present the mathematical formulation of DFPT and we
detail the links between the orbitals' response $\dphi_n$ and the ground-state
density response $\drho$ for a given external perturbation, as well as
the derivation of the Sternheimer equation \eqref{eq:Q_intro}.
In \autoref{sec:response}, we propose
a common framework for different natural gauge choices. Then, with
focus on the Sternheimer equation and the Schur complement,
we present the improved resolution to obtain $\dphi_n^Q$. Finally, in
\autoref{sec:num}, we perform numerical simulations on relevant physical
systems.  In the appendix, we propose a strategy for choosing
the number of extra orbitals motivated by a rough convergence analysis
of the Sternheimer equation.

\section{Mathematical framework}\label{sec:math}

\subsection{Periodic Kohn-Sham equations}

We consider here a simulation cell $\Omega=[0,1) \bm a_1 + [0,1) \bm a_2 +
[0,1) \bm a_3$ with periodic boundary conditions, where $(\bm a_1,\bm a_2, \bm
a_3)$ is a nonnecessarily orthonormal basis of
$\R^3$.
We denote by $\Rc=\Z \bm a_1 + \Z \bm a_2 + \Z \bm a_3$ the periodic lattice in
the position space and by $\Rc^*~=~\Z \bm b_1 + \Z \bm b_2 + \Z \bm b_3$ with
$\bm a_i\cdot \bm b_j = 2\pi\delta_{ij}$ the reciprocal lattice.    Let us
denote by
\begin{equation}
  \L^2_\#(\R^3, \C) \coloneqq \{u \in \L^2_{\rm loc}(\R^3,\C) \; | \; u \; \text{ is
  } \Rc\mbox{-periodic} \}
\end{equation}
the Hilbert space of complex-valued $\Rc$-periodic locally square integrable
functions on $\R^3$, endowed with its usual inner product $\cro{\cdot,\cdot}$ and by
$\H^s_\#(\R^3,\C)$ the $\Rc$-periodic Sobolev space of order $s\in\R$
$$
\H^s_\#(\R^3,\C)\coloneqq\left\{ u = \sum_{\bm G \in \Rc^*} \widehat u_{\bm G} e_{\bm G}, \; \sum_{\bm G \in \Rc^*} (1+|\bm G|^2)^s |\widehat u_{\bm G}|^2 < \infty\right\}
$$
where $e_{\bm G}(\bm r) = \e^{\i\bm G\cdot \bm r}/\sqrt{\abs{\Omega}}$
is the Fourier mode with wave-vector $\bm G$.

\medskip

In atomic units, the KS equations for a system of $\Nel=2N_{\rm p}$ spin-unpolarized electrons at
finite temperature read
\begin{equation}\label{eq:KS}
  H_\rho \phi_n = \varepsilon_n \phi_n, \quad
  \varepsilon_1\lq \varepsilon_2\lq\cdots,\quad
  \cro{\phi_n,\phi_m} = \delta_{nm}, \quad \rho(\bm r)
  = \sum_{n=1}^{+\infty} f_n\abs{\phi_n(\bm r)}^2, \quad
  \sum_{n=1}^{+\infty} f_n = \Nel,
\end{equation}
where $H_\rho$ is the Kohn-Sham Hamiltonian. It is given by
\begin{equation}
  H_\rho = - \frac 1 2 \Delta + V + V_{\rho}^{\rm Hxc}
\end{equation}
where $V$ is the potential generated by the nuclei (or the ionic cores if
pseudopotentials are used) of the system, and
$
V_{\rho}^{\rm Hxc}(\bm r) = V^{\rm H}_\rho(\bm r) + V^{\rm xc}_\rho(\bm r)
$
is an $\Rc$-periodic real-valued function depending on $\rho$. The Hartree
potential $V^{{\rm H}}_\rho$ is the unique zero-mean solution to the
periodic Poisson equation
$- \Delta V^{{\rm H}}_\rho(\bm r) = 4\pi \left( \rho(\bm r) -
  \frac{1}{\abs{\Omega}}\int_\Omega \rho \right)$
and the function $V^{{\rm xc}}_\rho$ is the exchange-correlation
potential. $H_\rho$ is a self-adjoint operator on
$\L^2_\#(\R^3,\C)$, bounded below and with compact resolvent. Its spectrum is
therefore composed of a nondecreasing sequence of eigenvalues
$(\varepsilon_n)_{n\in\N^*}$ converging to $+\infty$.
Since $H_\rho$ depends on the electronic density $\rho$, which in turn depends
on the eigenfunctions $\phi_n$, \eqref{eq:KS} is a nonlinear eigenproblem,
usually solved with \emph{self-consistent field} (SCF) algorithms. These
algorithms are based on successive partial diagonalizations of the Hamiltonian $H_{\rho_n}$ built from the current iterate $\rho_n$.
See \cite{cancesConvergenceAnalysisDirect2021,
  cancesNumericalMethodsKohnSham2021,linMathematicalIntroductionElectronic2019}
and references therein for a mathematical presentation of SCF algorithms.

\medskip

In \eqref{eq:KS}, the $\phi_n$'s are the
Kohn-Sham orbitals, with energy $\varepsilon_n$ and occupation number $f_n$.
At finite temperature $T>0$, $f_n$ is a real number in the interval $[0,2)$ and
we have
\begin{equation}\label{eq:fn}
  f_n = f\prt{\frac{\varepsilon_n - \eF}{T}},
\end{equation}
where $f$ is a fixed analytic \emph{smearing} function, which we choose here equal to
twice the Fermi-Dirac function: $f(x) = 2 / (1 + \e^x )$.
The Fermi level $\eF$ is then uniquely defined by the charge constraint
\begin{equation}\label{eq:eF}
  \sum_{n=1}^{+\infty} f_n = \Nel.
\end{equation}
When $T\to0$,
$f\prt{(\cdot - \eF)/T} \to 2\times\bm1_{\set{\cdot<\eF}}$ almost everywhere, and only the lowest $\Np=\Nel/2$ energy levels for which
$\varepsilon_n<\eF$ are occupied by two electrons of opposite spins (see
\autoref{fig:Fermi-Dirac}): $f_n=2$ for $n\lq \Np$ and $f_n=0$ for $n>\Np$.

\begin{remark}[The case of perfect crystals]\label{rmk:Brillouin}
Using a finite simulation cell $\Omega$ with periodic boundary conditions is
usually the best way to compute the bulk properties of a material in the
condensed phase. Indeed, KS-DFT simulations are limited to, say $10^3-10^4$
electrons, on currently available standard computer architectures. Simulating
{\it in vacuo} a small sample of the material containing, say $10^3$ atoms,
would lead to completely wrong results, polluted by surface effects since about
half of the atoms would lay on the sample surface. Periodic boundary conditions are a
trick to get rid of surface effects, at the price of artificial interactions
between the sample and its periodic image. In the case of a perfect crystal with
Bravais lattice ${\mathbb L}$ and unit cell $\omega$, it is natural to choose a
periodic simulation (super)cell $\Omega =L \omega$ consisting of $L^3$ unit cells (we
then have $\Rc=L\mathbb L$). In the absence of spontaneous symmetry breaking,
the KS ground-state density has the same $\mathbb L$-translational invariance as the
nuclear potential. Using Bloch theory, the supercell eigenstates $\phi_{n}$ can then be relabelled as
$\phi_{n}(\bm r) = e^{i\bm k \cdot \bm r} u_{j\bm k}(\bm r)$, where
$u_{j\bm k}$ now has cell periodicity, and equations \eqref{eq:KS}--\eqref{eq:fn} can be rewritten
as
\begin{align}
  &   H_{\rho,\bm k} u_{j\bm k} = \varepsilon_{j\bm k} u_{j\bm k}, \quad
  \varepsilon_{1\bm k}\lq \varepsilon_{2\bm k}\lq\cdots,\quad
  \cro{u_{j\bm k},u_{j'\bm k}} = \delta_{jj'}, \\
  &\rho(\bm r)
  = \frac{1}{L^3} \sum_{\bm k \in {\mathcal G}_L} \sum_{j=1}^{+\infty} f_{j\bm
    k}\abs{u_{j\bm k}(\bm r)}^2, \quad
 \frac{1}{L^3} \sum_{\bm k \in {\mathcal G}_L} \sum_{j=1}^{+\infty} f_{j\bm k} = \Nel, \quad
 f_{j\bm k} = f\prt{\frac{\varepsilon_{j\bm k} - \eF}{T}} \\
 & H_{\rho, \bm k} = \frac12(-\i\nabla + \bm k)^2 + V + V^{{\rm Hxc}}_\rho,
\end{align}
where ${\mathcal G}_L=L^{-1} {\mathbb L}^* \cap \omega^*$. Here ${\mathbb L}^*$
is the dual lattice of $\mathbb L$ and $\omega = \R^3 / {\mathbb L}^*$ the first
Brillouin zone of the crystal. In the thermodynamic limit $L \to \infty$, we
obtain the periodic Kohn-Sham equations at finite temperature
\begin{align}
  &   H_{\rho,\bm k} u_{j\bm k} = \varepsilon_{j\bm k} u_{j\bm k}, \quad
  \varepsilon_{1\bm k}\lq \varepsilon_{2\bm k}\lq\cdots,\quad
  \cro{u_{j\bm k},u_{j'\bm k}} = \delta_{jj'}, \\
  &\rho(\bm r)
  = \fint_{\omega^*} \sum_{j=1}^{+\infty} f_{j\bm k}\abs{u_{j\bm k}(\bm
    r)}^2 \d\bm k,
  \quad
 \fint_{\omega^*} \sum_{j=1}^{+\infty} f_{j\bm k} \d \bm k = \Nel, \quad
 f_{j\bm k} = f\prt{\frac{\varepsilon_{j\bm k} - \eF}{T}}. 
\end{align}
This is a massive reduction in complexity, as now only computations on
the unit cell have to be performed. For metals, the integrand on the Brillouin zone is discontinuous at zero
temperature, which makes standard quadrature methods fail. Introducing a
smearing temperature $T > 0$ allows one to smooth out the integrand, see
\cite{cancesNumericalQuadratureBrillouin2020,
  levittScreeningFiniteTemperatureReduced2020} for a numerical analysis of the
smearing technique. We also refer for instance to \cite[Section
XIII.16]{reedAnalysisOperators1978} for more details on Bloch theory, to
\cite{cancesNewApproachModeling2008,cattoMathematicalTheoryThermodynamic1998}
for a proof of the thermodynamic limit for perfect crystals in the rHF setting
for both insulators and metals, and to \cite{gontierConvergenceRatesSupercell2016} for
the numerical analysis for insulators.
\end{remark}

\subsection{Density-functional perturbation theory}

We detail in this section the mathematical framework of DFPT. We first rewrite the Kohn-Sham equations \eqref{eq:KS} as
the fixed-point equation for the density $\rho$
\begin{equation}\label{eq:ptv}
  F\prt{V+V^{\rm Hxc}_\rho} = \rho,
\end{equation}
where $F$ is the potential-to-density mapping defined by
\begin{equation}
  F(V)(\bm r) =
  \sum_{n=1}^{+\infty} f\prt{\frac{\varepsilon_n - \eF}T}\abs{\phi_n(\bm r)}^2
\end{equation}
with $(\varepsilon_n,\phi_n)_{n\in\N^*}$ an orthonormal basis of eigenmodes
of $-\frac12\Delta + V$ and $\eF$ defined by \eqref{eq:eF}. The
solution of \eqref{eq:ptv} defines a mapping from $V$ to $\rho$: the
purpose of DFPT is to compute its derivative. Let $\dH_0$ be a
local infinitesimal perturbation, in the sense that it can be represented by a multiplication
operator by a periodic function $\bm r \mapsto \delta V_0(\bm r)$.
By taking the
derivative of \eqref{eq:ptv} with the chain rule, we obtain the implicit
equation for $\drho$:
\begin{equation}
  \drho = F'\prt{V + V^{\rm Hxc}_\rho}\cdot(\dH_0 + K^{\rm
    Hxc}_\rho\drho),
\end{equation}
where the Hartree-exchange-correlation kernel $K^{\rm Hxc}_\rho$ is the
derivative of the map $\rho\mapsto V^{\rm Hxc}_\rho$ and
$F'\prt{V + V^{\rm Hxc}_\rho}$ is the derivative of $F$ computed at
$V + V^{\rm Hxc}_\rho$. In the field of DFT calculations, the latter operator is known
as the \emph{independent-particle susceptibility} operator and is denoted by $\chi_0$.
This yields the Dyson equation
\begin{equation}\label{eq:DFPT_tot}
  \drho = \chi_0(\dH_0 + K^{\rm Hxc}_\rho\drho)
  \quad\Leftrightarrow\quad
  \drho = \prt{1 - \chi_0K^{\rm Hxc}_\rho}^{-1}\chi_0\dH_0.
\end{equation}
This equation is commonly solved by iterative methods, which require efficient and robust
computations of $\chi_0\dH$ for various right-hand sides $\dH$'s. In the rest of this
article, we forget about the solution of \eqref{eq:DFPT_tot} and
focus on the computation of the noninteracting response $\drho \coloneqq \chi_0\dH$ for
a given $\dH$.

\medskip

The operator $\chi_0$ maps $\dH$ to the first-order variation $\drho$
of the ground-state density of a noninteracting system of electrons ($K^{\rm Hxc} = 0$).
Denoting $A_{mn} \coloneqq \cro{\phi_m,A\phi_n}$ for a
given operator $A$, it holds
\begin{equation}\label{eq:reponse}
  \drho(\bm{r}) = \sum_{n=1}^{+\infty} \sum_{m=1}^{+\infty}
  \frac{f_{n}-f_{m}}{\varepsilon_{n}-\varepsilon_{m}} \phi_n^*(\bm{r})\phi_m(\bm{r})
  \prt{\dH_{mn} - \deF\delta_{mn}},
\end{equation}
where $\delta_{mn}$ is the Kronecker delta, $\deF$ is the induced variation in
the Fermi level and we use the following convention
\begin{equation}\label{eq:conv}
  \frac{f_n - f_n}{\varepsilon_n - \varepsilon_n} =
  \frac{1}{T}f'\prt{\frac{\varepsilon_n - \eF}{T}} \eqqcolon f'_n.
\end{equation}
Charge conservation leads to
\begin{equation}\label{eq:deF}
  \int_\Omega \drho(\bm r)\d\bm r = 0 \quad\Rightarrow\quad
  \deF = \frac{\sum_{n=1}^{+\infty}
    f'_n\delta\varepsilon_n}{\sum_{n=1}^{+\infty}f'_n},
\end{equation}
where $\delta\varepsilon_n \coloneqq \dH_{nn}$. We refer to
\cite{baroniPhononsRelatedCrystal2001} for a physical discussion of
this formula, and to \cite{cancesDielectricPermittivityCrystals2010,
  herbstBlackboxInhomogeneousPreconditioning2020,
  levittScreeningFiniteTemperatureReduced2020},
  where it is proven rigorously using contour integrals.

\begin{remark}
  Similar to the discussion above on the computation of perfect
  crystal employing Bloch theory, response computations of perfect
  crystals can be performed by decomposing $\delta V_{0}$ in its Bloch
  modes. This allows for the efficient computation of phonon spectral
  or dielectric functions for instance.
\end{remark}
\begin{remark}
  We restricted our discussion for simplicity to local potentials, but
  the formalism can easily be extended to
  nonlocal perturbations (such as the ones created by pseudopotentials
  in the Kleinman-Bylander form
  \cite{kleinmanEfficaciousFormModel1982}).
\end{remark}

\subsection{Planewave discretization and numerical resolution}
\label{sec:PW_discr_DFPT}

In this paper we are interested in plane-wave DFT calculations of
metallic systems. This corresponds to a specific Galerkin approximation of the
Kohn-Sham model using as variational approximation space
\begin{equation}\label{eq:XN}
  X_\Nb \coloneqq \Span\set{e_{\bm G},\ \bm G\in\Rc^*,\ \frac12\abs{\bm G}^2 \lq \Ecut},
\end{equation}
where $\Nb$ denotes the dimension of the
discretization space, linked to the cut-off energy $\Ecut$.
Denoting by $\Pi_\Nb$ the orthogonal projection onto $X_\Nb$ for the $\L^2_\#$ inner product,
we then solve the discrete problem: find
$\phi_1,\dots,\phi_\Nb\in X_\Nb$ such that
\begin{equation}
  \label{eq:KS_discr}
  \begin{cases}
    \Pi_\Nb H_{\rho} \Pi_\Nb \phi_n = \varepsilon_n\phi_n,
    \quad \varepsilon_{1} \lq \cdots \lq \varepsilon_{\Nb},\\
    \rho = \sum_{n=1}^\Nb f_n\abs{\phi_n}^2,\quad \sum_{n=1}^\Nb f_n =
    \Nel,\quad f_n = f\prt{\frac{\varepsilon_n - \eF}T},\\
    \cro{\phi_n, \phi_{m}} = \delta_{nm}, \quad
    n,m=1,\dots,\Nb,
  \end{cases}
\end{equation}
where $H_\rho$ is the Kohn-Sham Hamiltonian (or one of its Bloch fibers).
This discretization method for
Kohn-Sham equations has been analysed for instance in
\cite{cancesNumericalAnalysisPlanewave2012}.

\medskip

We emphasize again the point
that not all $\Nb$ eigenpairs need to be computed.
At zero temperature, only the $N=\Nel/2$ lowest energy Kohn-Sham orbitals need to be
fully converged as they are the only occupied ones.
At finite temperature, the number of bands with meaningful occupation numbers
is usually higher than the number of electrons, but the fast
decay of the occupation numbers allows to avoid computing all $\Nb$ eigenpairs.
Determining the number of bands to compute is not easy as, at
finite temperature, we do not know \emph{a priori} the number of bands that are
significantly occupied. A standard heuristic is to fully converge 20\% more
bands than the number of electrons pairs during the SCF.
For the response calculation we then select the number $N$ of bands that
have occupation numbers above some numerical threshold.
On top of these bands,
it is common in DFT calculations to add additional bands that are not fully
converged by the successive eigensolvers. The main advantages of introducing
these bands are: (i) they enhance the diagonalization procedure by increasing
the gap between converged and uncomputed bands and (ii) adding extra bands is
not very expensive when the diagonalization is performed with block-based
methods, such as the LOBPCG algorithm
\cite{knyazevOptimalPreconditionedEigensolver2001}.

\section{Computing the response}\label{sec:response}

\subsection{Practical implementation}

Using \eqref{eq:reponse} as it stands is not possible because of the large
sums. One possibility is to represent
$\drho$ through a collection of occupied orbital variations $(\dphi_n)_{1\lq n\lq N}$ and
occupation number variations $(\df_n)_{1\lq n\lq N}$. One then has to make
appropriate ansatz and gauge choices on the links between $\drho$ and its
representation. Differentiating the formula $\rho(\bm
r) = \sum_{n=1}^N f_n\abs{\phi_n(\bm r)}^2$, one gets
\begin{equation}\label{eq:drho_ansatz}
  \drho(\bm r) = \sum_{n=1}^{N} f_n
  \prt{\phi_n^*(\bm r)\dphi_n(\bm r) + \dphi_n^*(\bm r)\phi_n(\bm r)}
  + \df_n \abs{\phi_n(\bm r)}^2.
\end{equation}
Then, for $n\lq N$, we expand $f_n\dphi_n$ into the basis $(\phi_m)_{m\in\N}$.
Defining
\begin{align}
  \Gamma_{mn} \coloneqq \cro{\phi_m,f_n\dphi_n},
\end{align}
yields
\begin{equation}\label{eq:dphi}
  \Forall 1\lq n\lq N,\quad
  f_n\dphi_n = \sum_{m=1}^N \Gamma_{mn}\phi_m + f_n\dphi_n^Q
\end{equation}
where $\dphi_n^Q \coloneqq Q\dphi_n$ and $Q$ is the orthogonal projector onto
$\Span(\phi_m)_{N<m}$, the space spanned by the unoccupied orbitals. Plugging
\eqref{eq:dphi} into \eqref{eq:drho_ansatz}, we obtain, using symmetry between
$n$ and $m$,
\begin{equation}\label{eq:drho_ansatz2}
  \drho(\bm r) = \sum_{n,m=1}^N \phi_n^*(\bm r)\phi_m(\bm r)\prt{\Gamma_{mn} +
    \wb{\Gamma_{nm}}} + \sum_{n=1}^N \df_n\abs{\phi_n(\bm r)}^2
  + \sum_{n=1}^N 2f_n\Re\prt{\phi_n^*(\bm r) \dphi_n^Q(\bm r)}.
\end{equation}
A first gauge choice can be made here. Using again the charge conservation, we get
\begin{equation}
  0 = \int_\Omega \drho(\bm r)\d\bm r \quad\Rightarrow\quad
  0 = \sum_{n=1}^N \Re(\Gamma_{nn}) + \df_n.
\end{equation}
Given that we adapt $\df_n$ accordingly
we can thus assume $\Gamma_{nn}=0$ for any $1\lq n\lq N$.
We will make this gauge choice from this point,
leaving the constraint $\sum_{n=1}^N\df_n=0$
to restrict possible choices of $\df_n$.

\medskip

We now derive conditions on $(\Gamma_{mn})_{1\lq n,m\lq N}$,
$(\df_n)_{1\lq n\lq N}$ and $(\dphi_n^Q)_{1\lq n\lq N}$ so that
the ansatz we made is a valid representation of $\drho$, that is to say
\eqref{eq:drho_ansatz2} coincides with \eqref{eq:reponse}. To this end, we
rewrite \eqref{eq:reponse} as
\begin{equation}\label{eq:reponse2}
  \drho(\bm{r}) = \sum_{n,m=1}^{N}
  \frac{f_{n}-f_{m}}{\varepsilon_{n}-\varepsilon_{m}} \phi_n^*(\bm{r})\phi_m(\bm{r})
  \prt{\dH_{mn} - \deF\delta_{mn}}
  + \sum_{n=1}^N \sum_{m=N+1}^{+\infty} 2\frac{f_n}{\varepsilon_n-\varepsilon_m}
  \Re\prt{\phi_n^*(\bm r)\phi_m(\bm r)\dH_{mn}},
\end{equation}
where the terms $f_n$, $f_m$ for which $n,m>N+1$ have been neglected because of their small
occupation numbers and we used the symmetry between $n$ and $m$ for the terms
with $1\lq n\lq N$, $m>N$. From a term by term comparison between
\eqref{eq:drho_ansatz2} and \eqref{eq:reponse2},
we infer first from the $n=m$ term and the
gauge choice $\Gamma_{nn} = 0$ that $\df_n = f'_n\prt{\dH_{nn} - \deF} =
f'_n\prt{\delta\varepsilon_n - \deF} $. Note
that, thanks to the definition \eqref{eq:deF} of $\deF$, charge conservation is
indeed satisfied. Next, for the first sum to coincide between
\eqref{eq:drho_ansatz2} and \eqref{eq:reponse2}, we see that
the $\Gamma_{mn}$'s have to satisfy
\begin{equation}\label{eq:gamma}
  \Forall 1\lq n, m\lq N,m \neq n,\ \quad \Gamma_{mn} + \wb{\Gamma_{nm}} =
  \frac{f_{n}-f_{m}}{\varepsilon_{n}-\varepsilon_{m}} \dH_{mn} \eqqcolon
  \Delta_{mn}.
\end{equation}
Finally, since $\dphi_n^Q\in\Span(\phi_m)_{N<m}$, we deduce from the last sum
in \eqref{eq:drho_ansatz2} and \eqref{eq:reponse2}
that $\dphi_n^Q$ can be computed as the unique solution of the linear system
\begin{equation}\label{eq:sternheimer}
  \Forall 1\lq n\lq N,\quad Q(H_\rho - \varepsilon_n)Q\dphi_n^Q = -Q\dH\phi_n,
\end{equation}
sometimes known in DFT as the Sternheimer equation
\cite{sternheimerElectronicPolarizabilitiesIons1954}.
Note that $\dphi_N^{Q}$ can be arbitrarily large,
since $\varepsilon_{N+1} - \varepsilon_{N}$ may be arbitrarily small.
However, this does not pose a problem in practice
as $\dphi_N^{Q}$ is multiplied by $f_{N}$ (cf. \eqref{eq:drho_ansatz2}),
which is very small.

\medskip

To summarize, the response $\drho = \chi_0\dH$ can be computed as
\begin{equation}\label{eq:drhoN}
  \drho(\bm r) = \sum_{n=1}^{N} 2f_n
  \Re\prt{\phi_n^*(\bm r)\dphi_n(\bm r)} + \df_n \abs{\phi_n(\bm r)}^2.
\end{equation}
Here, $\df_n = f'_n\prt{\delta\varepsilon_n - \deF}$, and $\dphi_n$ is separated into
two contributions:
\begin{equation}
  \Forall 1\lq n\lq N,\quad \dphi_n = \dphi_n^P + \dphi_n^Q,
\end{equation}
where $(\dphi_n^P,\dphi_n^Q)\in\Ran(P)\times\Ran(Q)$ with $P$ the orthogonal projector onto
$\Span(\phi_m)_{1\lq m\lq N}$ and $Q=1-P$.
These two contributions are computed as follows:
\begin{itemize}
  \item $\dphi_n^P$ is computed \emph{via} a sum-over-states $m\neq n$:
    \begin{equation}
      \dphi_n^P = \sum_{m=1,m \neq n}^N \Gamma_{mn}\phi_m,
    \end{equation}
    where the $\Gamma_{mn}$'s satisfy $\Gamma_{mn} + \wb{\Gamma_{nm}} =
    \Delta_{mn}$.
    An additional gauge choice has to be
    made as these constraints do not yet define $\Gamma_{nm}$ uniquely.
    We refer to this term as the occupied-occupied contribution.
  \item $\dphi_n^Q$ is obtained as the solution of the Sternheimer equation
    \eqref{eq:sternheimer}. However, this linear system is possibly very
    ill-conditioned if $\varepsilon_{N+1} - \varepsilon_N$ is small. We refer to
    this term as the unoccupied-occupied contribution.
\end{itemize}
Note that, at zero temperature, $\dphi_n^P$ vanishes so that $\dphi_n =
\dphi_n^Q \in \Ran(Q)$ and only the Sternheimer equation
\eqref{eq:sternheimer} needs to be solved.
In the next two sections, we detail the practical computation of these two
contributions.

\subsection{Occupied-occupied contributions}
In this section we discuss possible gauge choices for $\Gamma_{mn}$
to obtain a unique solution to \eqref{eq:gamma}.
Throughout this section we assume $m \neq n$ and $\Gamma_{nn} = 0$.

\subsubsection{Orthogonal gauge}
The orthogonal gauge choice is motivated from the zero temperature setting,
where $\dphi_{n}^{P} = 0$ allows to trivially preserve the
orthogonality amongst the computed orbitals $\phi_n$ under the perturbation.
For the case involving temperature, we additionally impose
\begin{equation}
  0 = \delta \cro{\phi_m,\phi_n} = \cro{\phi_m,\dphi_n} + \cro{\dphi_m,\phi_n},
\end{equation}
and therefore
\begin{equation}
  \frac1{f_n}\Gamma_{mn} + \frac1{f_m}\wb{\Gamma_{nm}} = 0,
\end{equation}
yielding
\begin{equation}
  \Gamma_{mn}^{\rm orth} = \frac{f_n}{\varepsilon_n - \varepsilon_m}\dH_{mn}
  \text{ for } m \neq n.
\end{equation}
As a result $f_n \dphi_n$ features a
possibly large contribution $\Gamma_{mn}^{\rm orth}$,
which is going to be almost compensated in \eqref{eq:drhoN} by
the large contribution $\wb{\Gamma_{nm}^{\rm orth}}$ to $f_n \dphi_n^\ast$
due the requirement to sum to the moderate-size contribution
$\Gamma_{mn}^{\rm orth} + \wb{\Gamma_{nm}^{\rm orth}} = \Delta_{mn}$.
This can lead to numerical instabilities because small errors,
\eg~due to the fact that the $\phi_{n}$'s in \eqref{eq:drhoN} are eigenvectors
only up to the solver tolerance, will get amplified by the $\Gamma_{mn}$.
The next gauge choices provide solutions to this issue.

\subsubsection{Simple gauge choice}
Possibly the simplest gauge choice is $\Gamma_{mn}^{\rm simple} =
\frac12\Delta_{mn}$. Since $(\Delta_{mn})_{1\lq n,m\lq N}$ is Hermitian,
\eqref{eq:gamma} is immediately satisfied.

\subsubsection{Quantum Espresso gauge}

The DFPT framework presented in \cite{baroniPhononsRelatedCrystal2001} and
implemented in Quantum Espresso \cite{giannozziQUANTUMESPRESSOModular2009}
suggests choosing
\begin{equation}
  \Gamma_{mn}^{\rm QE} = f_{\rm FD}\prt{\frac{\varepsilon_n - \varepsilon_m}{T}} \Delta_{mn},
\end{equation}
where $f_{\rm FD} = \frac12 f$ is the Fermi-Dirac functional.
Since $f_{\rm FD}(x) + f_{\rm FD}(-x)=1$, we have $\Gamma_{mn}^{\rm QE} +
\wb{\Gamma_{mn}^{\rm QE}}=\Delta_{mn}$.

\subsubsection{Abinit gauge} In the Abinit software
\cite{gonzeAbinitProjectImpact2020,romeroABINITOverviewFocus2020},
the choice is
\begin{equation}
  \Gamma_{mn}^{\rm Ab} = \bm{1}_{\{f_n>f_m\}}\Delta_{mn}.
\end{equation}

\subsubsection{Minimal gauge}
Motivated by our analysis of the instabilities we suggest minimizing $\dphi_n$,
that is to ensure $\Gamma_{mn}/f_n$ to stay as small as possible.
This leads to the minimization problem
\begin{equation}
  \begin{aligned}
      &&\min\hspace{10pt} &\displaystyle \sum_{n,m=1,m \neq n}^{N}
    \frac{1}{f_n^2} \abs{\Gamma_{mn}}^2,\\
    &&\text{s.t.}\hspace{10pt}&\Gamma_{mn} + \wb{\Gamma_{nm}} = \Delta_{mn}, \quad\Forall
    1\lq n,m\lq N,\ m \neq n.
  \end{aligned}
\end{equation}
As the constraint \eqref{eq:gamma}
only couples $(n,m)$ and $(m,n)$, this translates into an uncoupled system of
constrained minimization problems: for $1\lq n,m\lq N$, $m \neq n$, solve
\begin{equation}
  \begin{aligned}
      &&\min\hspace{10pt} &\displaystyle \frac{1}{f_n^2} \abs{\Gamma_{mn}}^2 + \frac{1}{f_m^2} \abs{\wb{\Gamma_{nm}}}^2,\\
    &&\text{s.t.}\hspace{10pt} & \Gamma_{mn} + \wb{\Gamma_{nm}} = \Delta_{mn},
  \end{aligned}
\end{equation}
whose solution is
\begin{equation}\label{eq:min_gauge}
  \Gamma_{mn}^{\rm min} = \frac{f_n^2}{f_n^2 + f_m^2}\Delta_{mn}.
\end{equation}
This gauge choice is implemented by default in the DFTK software
\cite{herbstDFTKJulianApproach2021}. Another gauge choice inspired from this one
would be to directly minimize $\abs{\Gamma_{mn}}^2 + \abs{\wb{\Gamma_{nm}}}^2$
but it can be shown that this leads to the simple gauge choice $\Gamma_{mn}^{\rm simple} =
\frac12\Delta_{mn}$.

\subsubsection{Comparison of gauge choices}

From \eqref{eq:reponse} we can see that the growth of $\drho$ with respect to $\dH$
can not be higher than the growth of $\Delta_{mn}$ with respect to $\dH$.
The latter is of the order of $\max_{x\in\R} \frac{1}{T}\abs{f'(x)} = \frac{1}{2T}$,
which thus provides an intrinsic limit to the conditioning of the problem.
For all gauge choices but the orthogonal one easily verifies
\begin{equation}
  \abs{\Gamma_{mn}} \lq \abs{\Delta_{mn}} \lq
  \max_{x\in\R} \frac 1 T \abs{f'(x)} \abs{\dH_{mn}} = \frac{1}{2T} \abs{\dH_{mn}}.
\end{equation}
If we make an error on $\dH$ it is thus at most amplified
by a factor of $\frac{1}{2T}$.
All choices but the orthogonal one thus manage to stay within
the intrinsic conditioning limit, see \autoref{fig:comp_gauges}.
\begin{figure}
  \includegraphics[height=0.4\linewidth]{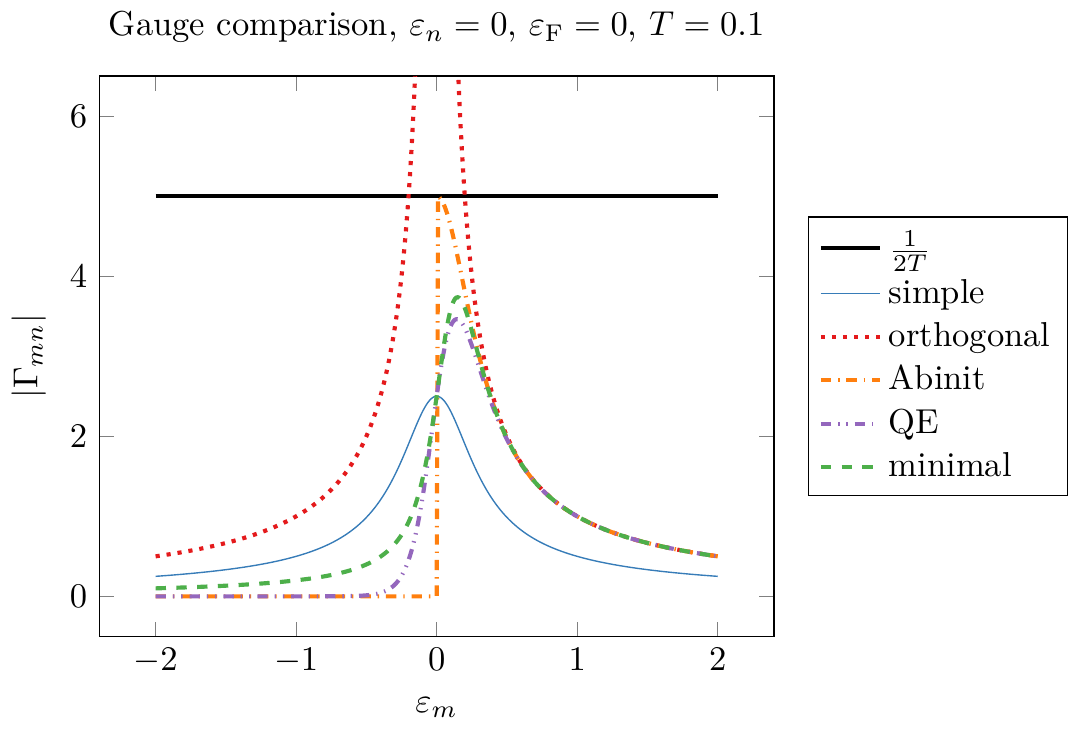}
  \caption{Comparison of gauge choices for $\dH_{mn}=1$.
    Except the orthogonal gauges, all contributions $\Gamma_{mn}$ are bounded by
    $\frac 1 {2T}$.}
  \label{fig:comp_gauges}
\end{figure}

\subsection{Computation of unoccupied-occupied contributions employing a Schur complement}
\label{sec:schur}

Since the $\phi_m$ for $m>N$ are not exactly known,
a different approach is needed for obtaining the contribution $\dphi_n^Q$.
Usually one resorts to solving the Sternheimer equation
\begin{equation}\label{eq:sternheimer1}
  \Forall 1\lq n\lq N,\quad Q(H_\rho - \varepsilon_n)Q\dphi_n^Q = -Q\dH\phi_n
  \eqqcolon b_n
\end{equation}
using iterative schemes restricted to $\Ran(Q)$.
However, for $n=N$ the difference $\varepsilon_{N+1} - \varepsilon_{N}$
can become small, which deteriorates conditioning and increases
the number of iterations required for convergence.

\medskip

We overcome this issue by making use of the $\Nex$ extra bands,
which are anyway available after the SCF algorithm has completed.
Following \autoref{sec:PW_discr_DFPT}
the $\Nex$ extra bands can be divided into two categories:
\begin{enumerate}
\item Some (usually the lower-energy ones)
      have been discarded during the response calculation
      because they have a too small occupation.
      Up to the eigensolver tolerance these are exact eigenvectors.
  \item The remaining ones have served to accelerate the successive
      diagonalization steps during the SCF.
      These have not yet been fully converged.
\end{enumerate}
In any case these extra bands thus offer (at least) approximate information
about some $\phi_m$ for $m>N$,
which is the underlying reason why the following approach accelerates the
computation of $\dphi_n^Q$.

\medskip

For the sake of clarity, we place ourselves here in the discrete setting:
$H_\rho\in\C^{\Nb\times\Nb}$, $\Phi\in\C^{\Nb\times N}$ and $\wt\Phi~\in~\C^{\Nb\times\Nex}$.
We assume that the number of computed bands
$N + \Nex$ is larger than the number of occupied states $N$
and that we trust $\Phi = (\phi_{1},\dots,\phi_{N})$ but not
$\wt{\Phi} = (\wt\phi_{N+1},\dots,\wt\phi_{N+\Nex})$ to be eigenvectors.
These $\Nex$ extra bands consist of both contributions (1) and (2) described at
the beginning of this section.
We assume in addition that $(\Phi,\wt{\Phi})$ forms an orthonormal family and
that $\wt\Phi^*H_\rho\wt\Phi$ is a diagonal matrix whose elements, denoted
by $(\wt\varepsilon_n)_{n=N+1,\dots,N+\Nex}$, are not necessarily all exact
eigenvalues. Note that Rayleigh-Ritz based iterative methods such as the LOBPCG
algorithm fit exactly in this framework.
\begin{figure}
  \centering
  \includegraphics[width=0.5\linewidth]{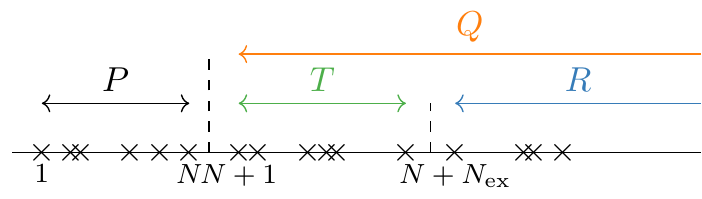}
  \caption{Graphical representation of the Schur decomposition to solve the
    Sternheimer equation. $P$ is the orthogonal projector onto the occupied
    states. $Q$ is the orthogonal projector onto the unoccupied states, and we
    decompose it as the sum of $T$ (extra states which we can use) and $R$
    (remaining states).}
  \label{fig:Q_decomp}
\end{figure}
We decompose
\begin{equation}
    \Ran(Q) = \Ran(T) \oplus \Ran(R),
\end{equation}
where $T$ is the orthogonal projector onto $\Span(\wt\phi_m)_{N<m\lq N+\Nex}$ and
$R=Q-T$ is the projector onto the remaining (uncomputed) states, see \autoref{fig:Q_decomp}.
Then, as $\dphi_n^Q\in\Ran(Q)$, we can decompose
\begin{equation}\label{eq:dphiQ_decomp}
  \dphi_n^Q = {\wt\Phi}\alpha_n + \dphi_n^R,
\end{equation}
where $\alpha_n \in \C^{\Nex}$ and $\dphi_n^R\in\Ran(R)$.
Plugging this into \eqref{eq:sternheimer1} we get
\begin{equation}
  \label{eq:plugged_stern}
  Q(H_\rho-\varepsilon_{n}){\wt\Phi}\alpha_n + Q(H_\rho-\varepsilon_{n})\dphi_n^R = b_n.
\end{equation}
Using a Schur complement we deduce
\begin{equation}\label{eq:alpha_n}
  \alpha_n = \prt{{\wt\Phi}^*(H_\rho-\varepsilon_n){\wt\Phi}}^{-1}\prt{{\wt\Phi}^*b_n -
    {\wt\Phi}^*(H_\rho-\varepsilon_n)\dphi_n^R}.
\end{equation}
Inserting \eqref{eq:alpha_n} into \eqref{eq:plugged_stern}
and projecting on $\Ran(R)$ yields an equation in $\dphi_n^R$:
\begin{equation}
  \label{eq:Schur}
  \begin{split}
    R(H_\rho-\varepsilon_n)\left[
      1 - {\wt\Phi}\prt{{\wt\Phi}^*(H_\rho-\varepsilon_n){\wt\Phi}}^{-1}{\wt\Phi}^*(H_\rho-\varepsilon_n)
    \right] R\dphi_n^R \\ =
    Rb_n - R(H_\rho-\varepsilon_n){\wt\Phi}\prt{{\wt\Phi}^*(H_\rho-\varepsilon_n){\wt\Phi}}^{-1}{\wt\Phi}^*b_n.
  \end{split}
\end{equation}
This equation can then be solved for $\dphi_n^R$ with a Conjugate Gradient
(CG) method which is enforced to stay in $\Ran(R)$ at each iteration.
Afterwards we compute $\alpha_n$ from \eqref{eq:alpha_n}, which yields $\dphi_n^Q$
from \eqref{eq:dphiQ_decomp}.
This scheme has been implemented as the default solver for the Sternheimer
equation in DFTK.

\section{Numerical tests}\label{sec:num}

For all the numerical tests, we use the DFTK
software~\cite{herbstDFTKJulianApproach2021}, a recent plane-wave DFT package
in \texttt{Julia}.  All the codes to run the simulation of this paper are
available
online\footnote{\url{https://github.com/gkemlin/response-calculations-metals}}.
The Brillouin zone is discretized using a uniform
Monkhorst-Pack grid \cite{monkhorstSpecialPointsBrillouinzone1976}. We use
the PBE exchange-correlation functional
\cite{perdewGeneralizedGradientApproximation1996} and GTH pseudopotentials
\cite{goedeckerSeparableDualspaceGaussian1996,hartwigsenRelativisticSeparableDualspace1998}.
The other parameters of the calculation will be specified for each example.
In all the tests, we generate a perturbation $\dH$ from atomic
displacements, with local and nonlocal contributions. Then, we perform two
response calculations: one with the standard approach to solve directly the Sternheimer
equation \eqref{eq:sternheimer1} to compute $\dphi_n^Q$, the other with
the (new) Schur complement approach \eqref{eq:Schur}. Both linear systems are solved using
the conjugate gradient (CG) algorithm, with kinetic energy preconditioning (the linear
solver is preconditioned with the inverse Laplacian, which is diagonal in
Fourier representation),
and we compare the number of iterations required to converge the norm of the residual below
$10^{-9}$. Note that the Sternheimer equation is solved for all $N$ occupied
orbitals and for each $k$-point.

\medskip
If $T>0$ the contribution $\dphi_n^P$ is nonzero
and has been computed using the sum-over-states formula
with the minimal gauge choice \eqref{eq:min_gauge}.
In terms of runtime we expect only negligible differences between the gauge choices.
Moreover, since the time for this contribution is much smaller compared
to the time required to solve the Sternheimer equation,
we do not report a detailed performance comparison on this step in the following.

\medskip
Note that our purpose in this paper is only to improve the numerical algorithms used in response computations. Although the parameters used here (exchange-correlation functionals, pseudopotentials, smearing and Brillouin zone sampling parameters, cut-off energies) might not represent physical reality appropriately, they are representative of practical calculations.

\subsection{Insulators and semiconductors}
\begin{figure}
  \centering
  \includegraphics[height=0.4\linewidth]{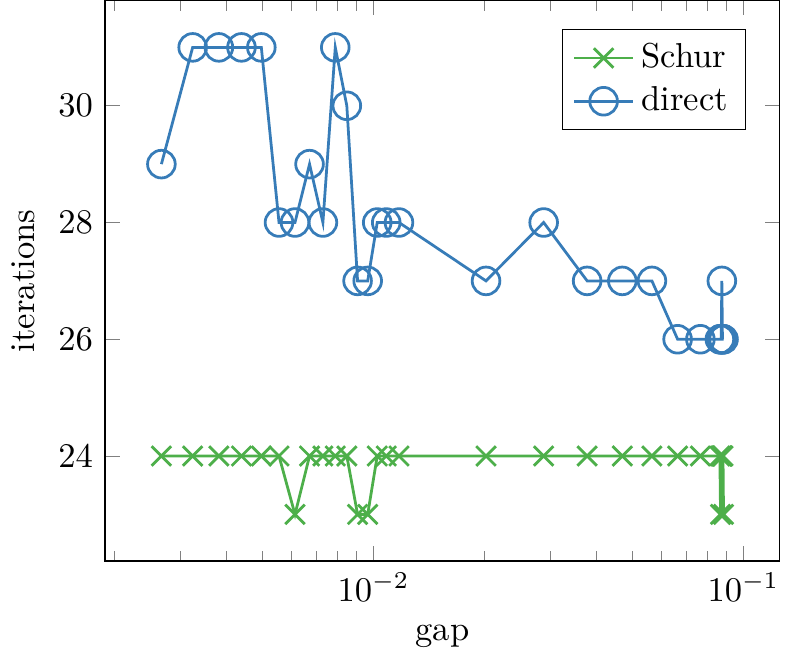}
  \caption{Number of iterations of the linear solver for the Sternheimer
    equation for $n=N=4$ \emph{vs} the gap, with and without the Schur complement
    \eqref{eq:Schur}.}
  \label{fig:schur_sternheimer_silicon}
\end{figure}

For insulators and semiconductors
the gap between occupied and virtual states is usually large.
One would therefore not expect a large gain from using the Schur complement
\eqref{eq:Schur} when computing $\dphi_n = \dphi_n^Q$.
However, for distorted semiconductor structures or semiconductors with defects
the gap can be made arbitrarily small,
such that one would expect to see the Schur-complement approach to be in the advantage.
We test this using an FCC Silicon crystal for which we increase the lattice constant
from 10 bohrs to $11.4$ bohrs
to artificially decrease and eventually close the gap.
All calculations have been performed using a cut-off energy
of $\Ecut = 50$ Ha and a single ${k}$-point (the $\Gamma$-point).
In \autoref{fig:schur_sternheimer_silicon} we plot the number of iterations
required for the linear solver of the Sternheimer equation
to converge, for $n=N=4$.  Using the Schur complement the number of iterations
stays almost constant even when the gap decreases.
In contrast, with a direct approach, the linear solver requires about $30\%$
more iterations near the closing gap.

\subsection{Metals}
The real advantage of using the Schur complement \eqref{eq:Schur} instead of
directly solving the Sternheimer equation \eqref{eq:sternheimer1} becomes apparent
when computing response properties for metals at finite temperature.
We use a standard heuristic which suggests to
fully converge 20\% more bands than the number of electrons pairs of the system, with
3 additional extra bands that are not converged by the successive eigensolvers
of the SCF. We then select the \enquote{occupied} orbitals with an occupation threshold of
$10^{-8}$. In addition to the number of iterations, we also compare the cost of
the response calculations with and without the Schur complement \eqref{eq:Schur}.
For this we consider the total number of Hamiltonian applications
which was required to compute the response $\drho$.
For the small to medium-sized systems we consider here,
the Hamiltonian-vector-product is the most expensive step in an DFT calculation
and thus provides a representative cost indicator.
Notice that both the implementation of the Schur complement and the direct
method require exactly one Hamiltonian application per iteration of the CG.
Additionally the Schur approach requires the computation of $H_\rho\wt\phi$,
which is only a negligible additional cost as this is only needed once per
$k$-point.

\subsubsection{Aluminium} We start by considering an elongated aluminium supercell
with 40 atoms.
We use a cut-off energy $\Ecut=40$ Ha,
a temperature $T=10^{-3}$ Ha with Fermi-Dirac smearing
and a $3\times3\times1$ discretization of the Brillouin zone.
To ensure convergence of the SCF iterations we employ
the Kerker preconditioner \cite{kerkerEfficientIterationScheme1981}.
Since the system has $120$ electrons per unit cell our usual heuristic
converges 72 bands up to the tolerance of the eigensolver
accompanied by 3 bands, which are not fully converged.

\medskip

\begin{figure}
  \centering
  \includegraphics[height=0.4\linewidth]{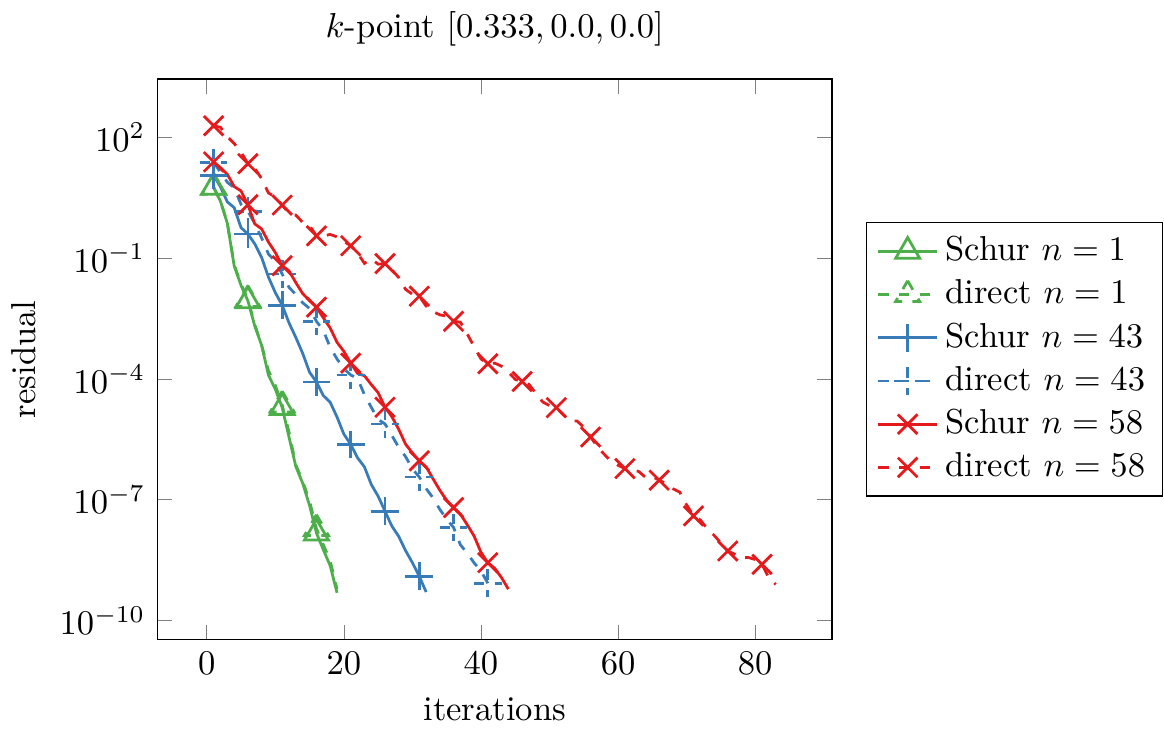}
  \caption{Convergence of the Sternheimer solver for three different orbitals
    for \ch{Al40}. Each curve represents the convergence of the CG which solves
    the Sternheimer  equation for one orbital:  those with the slowest
    convergence are associated to the occupied orbitals with the highest energy.}
  \label{fig:Al40}
\end{figure}
\begin{table}
  \begin{tabular}{@{}cccccccc@{}}
\toprule
\multicolumn{2}{c}{$k$-point -- coordinate}                           & \multicolumn{2}{c}{$1 - [0,0,0]$}         & \multicolumn{2}{c}{$2 - [1/3,0,0]$}         & \multicolumn{2}{c}{$5 - [1/3,1/3,0]$}        \\ \midrule
\multicolumn{2}{c}{$N$}                                 & \multicolumn{2}{c}{$69$}        & \multicolumn{2}{c}{$58$}        & \multicolumn{2}{c}{$67$}       \\
\multicolumn{2}{c}{$\varepsilon_{N+1} - \varepsilon_N$} & \multicolumn{2}{c}{$0.0320$}     & \multicolumn{2}{c}{$0.0134$}    & \multicolumn{2}{c}{$0.0217$}   \\
\multicolumn{2}{c}{\#iterations $n=N$ Schur}                &
\multicolumn{2}{c}{$48$}        & \multicolumn{2}{c}{$\bm{44}$} &
\multicolumn{2}{c}{$41$}       \\
\multicolumn{2}{c}{\#iterations $n=N$ direct}             &
\multicolumn{2}{c}{$56$}        & \multicolumn{2}{c}{$\bm{83}$} & \multicolumn{2}{c}{$58$}       \\ \bottomrule
\end{tabular}

  \caption{Convergence data for $k$-points 1, 2 and 5 for \ch{Al40}.
    Other $k$-points are not displayed but they all behave as one of these by
    symmetry.  $N$ is the number of occupied bands, for an occupation
    threshold of $10^{-8}$. }
  \label{tab:Al40}
\end{table}
The convergence behaviour when solving the Sternheimer equation
for $k$-points of particular interest
is shown in \autoref{tab:Al40} and \autoref{fig:Al40}.
As expected, for $k$-points with a small difference
$\varepsilon_{N+1}-\varepsilon_N$,
the Schur complement \eqref{eq:Schur} brings a noteworthy improvement
with roughly $50\%$ fewer iterations required to achieve convergence.
Since the system we consider here has numerous occupied bands
--- between $60$ and $70$ depending on the $k$-point
--- most bands already feature a well-conditioned Sternheimer equation.
Considering the cost for computing the total response,
the Schur approach therefore overall only achieves
a reduction by 17\% in the number of Hamiltonian applications,
from about $17,800$ (direct) to $14,800$ (with Schur).
However, it should be noted that this improvement essentially comes for free
as the extra bands are anyway provided by the SCF computation as a byproduct.

\subsubsection{Heusler system}
Next we study the response calculation of Heusler-type transition-metal alloys.
We focus mainly on the \ch{Fe2MnAl} system but other compounds,
such as the \ch{Fe2CrGa} and \ch{CoFeMnGa} alloy systems,
have been tested and similar results were obtained.
Heusler alloys are of considerable practical interest
due to their rich and unusual magnetic and electronic properties.
For instance,
\ch{Fe2MnAl} shows halfmetallic behaviour: the majority spin channel (denoted by
$\uparrow$) behaves like a metal whereas the minority spin channel
(denoted by $\downarrow$) behaves like an insulator as it has a vanishing
density of states at the Fermi level. See \cite{herbstRobustEfficientLine2022},
and reference therein, for more details as well as an analysis of the SCF
convergence on such systems.
For these systems we use a cut-off $\Ecut = 45$ Ha, a temperature $T=10^{-2}$ Ha
with Gaussian smearing
and a $13\times13\times13$ discretization of the Brillouin zone. The SCF was
converged using a Kerker preconditioner
\cite{kerkerEfficientIterationScheme1981}. Moreover, as we deal with a
spin-polarized system, the numerical simulation slightly differs.
The orbitals $\phi_{(n,\bm k)}^\sigma$ and
the occupation numbers $f_{(n,\bm k)}^\sigma$ depend on the spin orientation
$\sigma\in\{\uparrow,\downarrow\}$ and the $f_{(n,\bm k)}^\sigma$'s
belong to $[0, 1)$ instead of $[0, 2)$.
Furthermore we modify the heuristic to determine the number of bands to be computed:
\ch{Fe2MnAl} has $\Nel=50$ electrons per unit cell and we use
$25 + 0.2\times50 = 35$ fully converged bands per $k$-point,
complemented by 3 additional bands, which are not checked for convergence.

\medskip

\begin{table}
  \begin{tabular}{@{}cccccc@{}}
\toprule
\multicolumn{2}{c}{spin channel}                        & \multicolumn{2}{c}{$\uparrow$}  & \multicolumn{2}{c}{$\downarrow$} \\ \midrule
\multicolumn{2}{c}{$N$}                                 & \multicolumn{2}{c}{$28$}        & \multicolumn{2}{c}{$26$}         \\
\multicolumn{2}{c}{$\varepsilon_{N+1} - \varepsilon_N$} & \multicolumn{2}{c}{$0.0423$}    & \multicolumn{2}{c}{$0.0154$}     \\
\multicolumn{2}{c}{\#iterations $n=N$ Schur}                &
\multicolumn{2}{c}{$\bm{45}$}        & \multicolumn{2}{c}{$\bm{45}$}         \\
\multicolumn{2}{c}{\#iterations $n=N$ direct}             & \multicolumn{2}{c}{$\bm{86}$}        & \multicolumn{2}{c}{$\bm{103}$}        \\ \bottomrule
\end{tabular}

  \caption{Convergence data for the two spin channels of the $k$-point with
    reduced coordinates $[0.385,0.231,0.077]$ for \ch{Fe2MnAl}.
    $N$ is the number of occupied bands, for an occupation
    threshold of $10^{-8}$. }
  \label{tab:Fe2MnAl}
\end{table}
\begin{figure}
  \centering
  \includegraphics[height=0.4\linewidth]{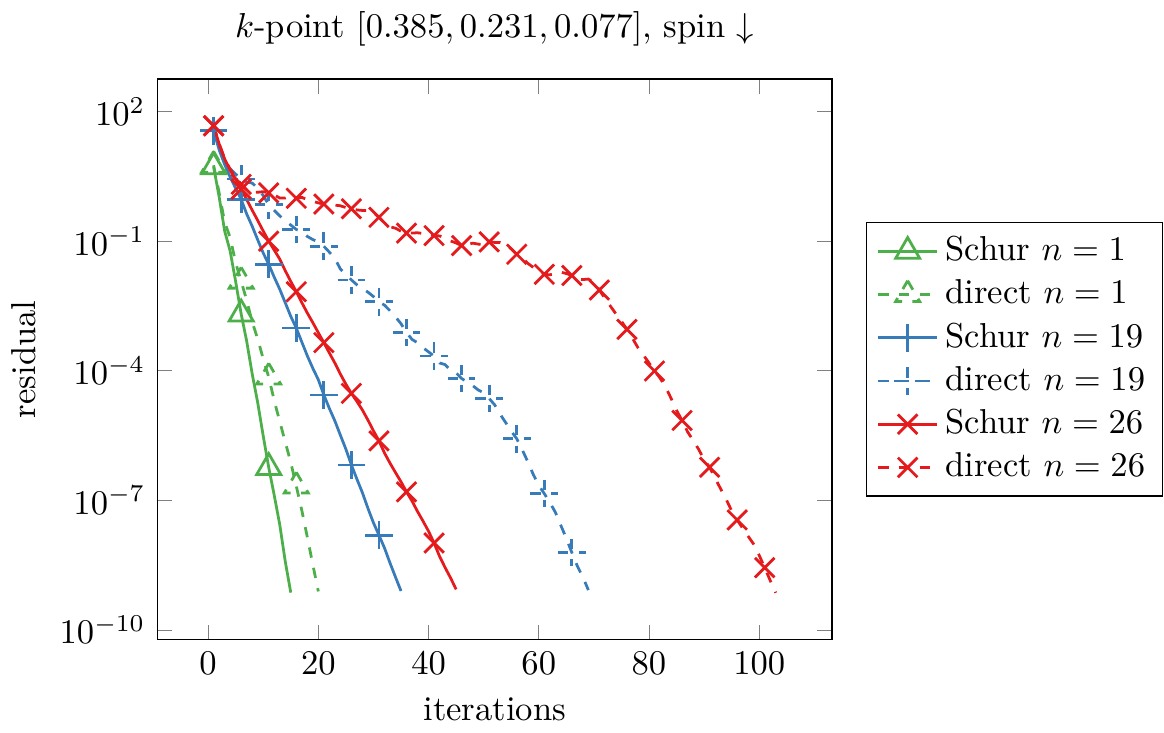}
  \caption{Convergence of the Sternheimer solver for three different orbitals
    for \ch{Fe2MnAl}. Each curve represents the convergence of the CG which solves
    the Sternheimer  equation for one orbital:  those with the slowest
    convergence are associated to the occupied orbitals with the highest energy.}
  \label{fig:Fe2MnAl}
\end{figure}
We show in \autoref{tab:Fe2MnAl} and \autoref{fig:Fe2MnAl} the results for the
two spin channels of the $k$-point with reduced coordinates
$[0.385,0.231,0.077]$. The other $k$-points behave similarly.
Since both channels feature a small difference
$\varepsilon_{N+1} - \varepsilon_N$ using the Schur complement
\eqref{eq:Schur} to solve the Sternheimer equation has a significant impact:
for the orbitals with highest energy it reduces the number of iterations by half.
For the direct approach we notice a plateau where the solver
encounters difficulties to converge the Sternheimer equation for the $N$-th orbital
due to the small gap.
\begin{figure}
  \centering
  \includegraphics[height=0.4\linewidth]{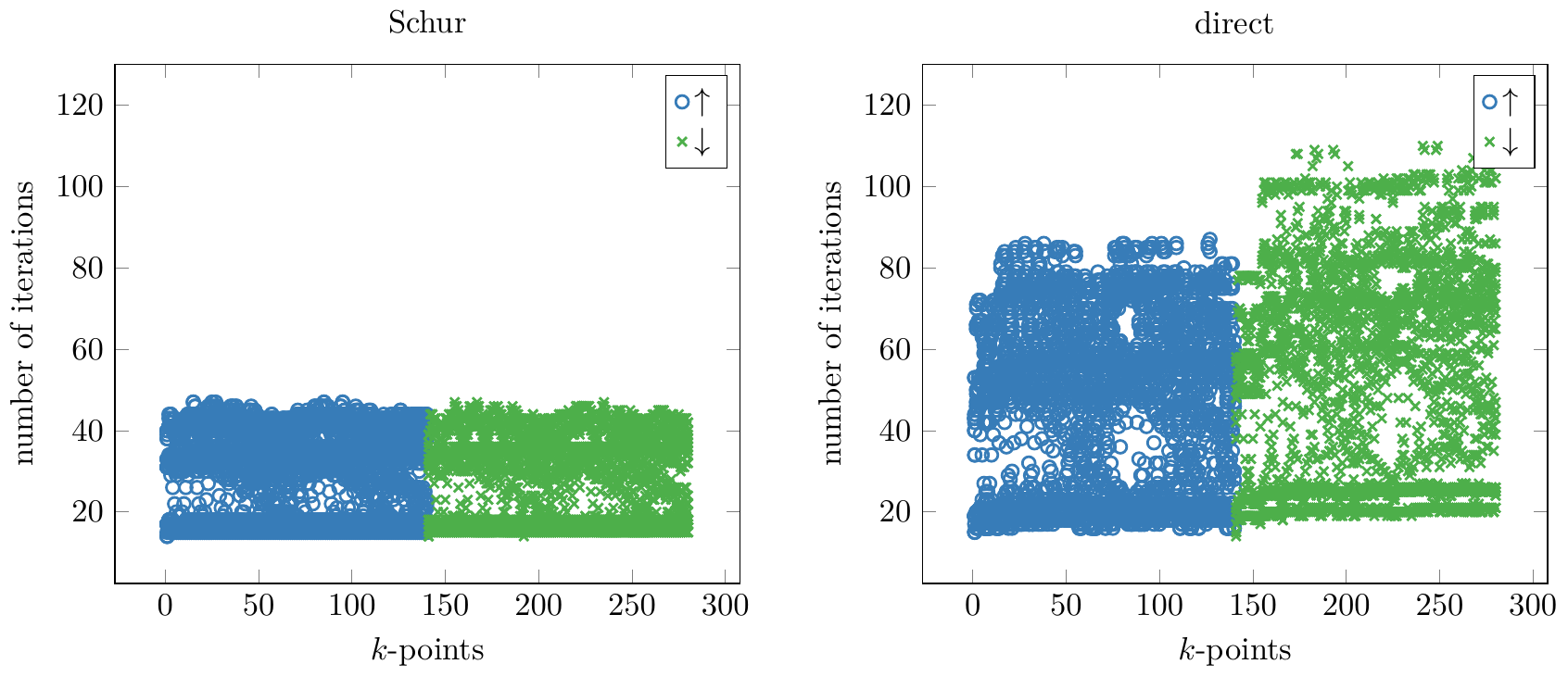}
  \caption{Histogram of the number of iterations of the CG to solve the
    Sternheimer equation, with and without the Schur complement
    \eqref{eq:Schur}. On the $x$-axis, the $k$-point index number: the first 140
    (blue $\circ$) have spins up, and the last 140 (green $\times$) have the same coordinates
    but with spins down. For each
    of these $k$-points, we plot the number of iterations for every occupied
    band of the $k$-point.}
  \label{fig:Fe2MnAl_histo}
\end{figure}
Unlike the aluminium case the improvements observed for the Heusler alloys
are not restricted to a small number of bands.
In \autoref{fig:Fe2MnAl_histo} we contrast the number of iterations
required to solve the Sternheimer equation for every band at every $k$-point
with and without using the Schur complement.
Notice that lattice symmetries allow to reduce the number of explicitly treated
$k$-points to $140$ albeit we are using a  $13\times13\times13$ $k$-point grid.
In terms of the total number of Hamiltonian applications required
for the response calculation, the Schur complement achieves
a reduction by roughly $40\%$, from around $344,000$ (without Schur)
to $208,000$ (with Schur).
It should be noted that in this system the standard heuristic
caused a large portion of the available extra bands to be fully converged,
thus providing an ideal setting for the Schur complement approach to be effective.
For example for the $k$-point discussed in \autoref{tab:Fe2MnAl}
seven extra bands have been fully converged and an additional three partially.
Given the enormous importance of Heusler systems and the known
numerical difficulties for computing response properties in these systems,
our result is encouraging and motivates the development of a more
economical heuristic for choosing the number of converged bands in future work.

\subsection{Comparison to shifted Sternheimer approaches}
In the literature other strategies for computing $\drho$ have been reported.
We briefly consider the approach proposed
in \cite{baroniPhononsRelatedCrystal2001}, where the response is computed as
\begin{equation}\label{eq:drho_baroni}
  \drho(\bm r) = \sum_{n=1}^N2\phi_n^*(\bm r)\dphi_n(\bm r) - f'_n
  \deF|\phi_n(\bm r)|^2.
\end{equation}
Instead of splitting $\dphi_n$ into two contributions,
the full $\dphi_n$ is computed for all $n\lq N$
by solving the \emph{shifted} Sternheimer equation
\begin{equation}\label{eq:shifted_sternheimer}
  (H_\rho + S - \varepsilon_n) \dphi_n = -(f_n - S_n)\dH\phi_n.
\end{equation}
Here $S:\Ran(P) \to \Ran(P)$ is a shift operator acting on the space of occupied
orbitals, chosen so that the linear system is nonsingular
(for any $n\lq N$, $H_\rho-\varepsilon_n$ is not invertible). Then, $S_n$
is chosen for every $n\lq N$ such that $\drho$ from  \eqref{eq:drho_baroni} satisfies
\eqref{eq:reponse}.
However, as $S$ only acts on $\Ran(P)$,
equation still becomes badly conditioned
if $\varepsilon_{N+1}- \varepsilon_N$ is too small.
This becomes apparent when solving the shifted Sternheimer equation
\eqref{eq:shifted_sternheimer} for the \ch{Fe2MnAl} system,
see \autoref{fig:shifted_sternheimer}.
\begin{figure}
  \includegraphics[width=0.48\linewidth]{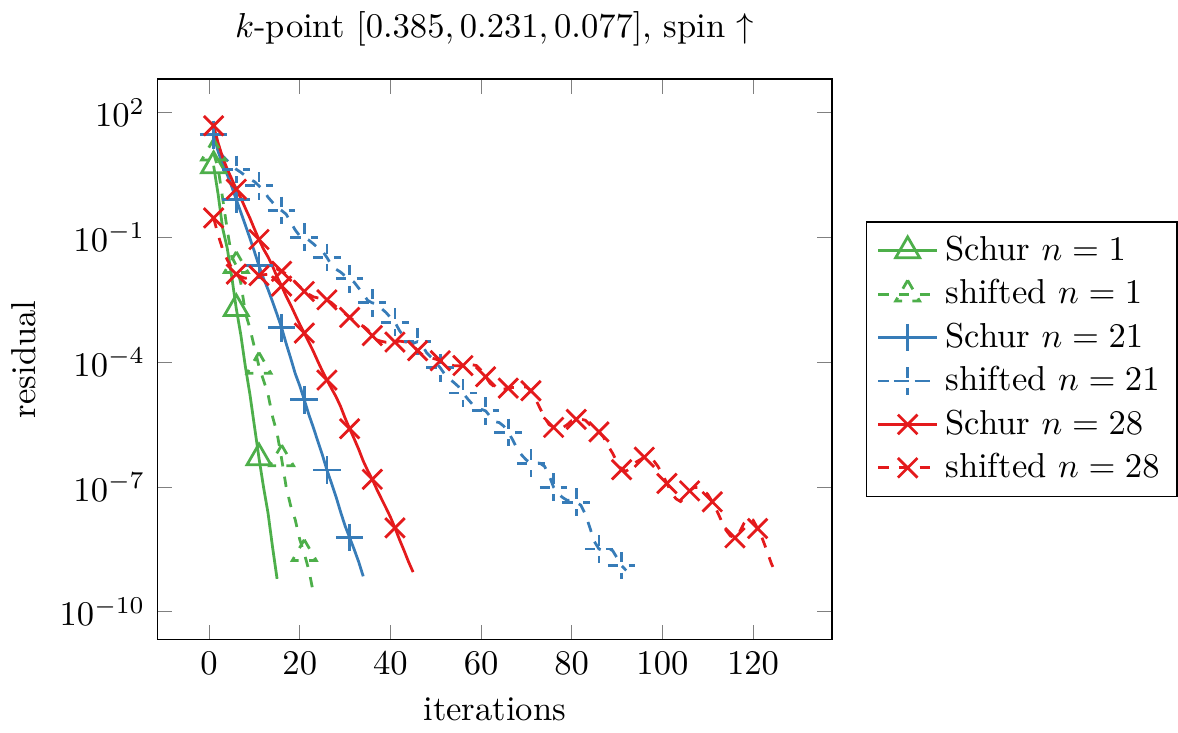}
  \includegraphics[width=0.48\linewidth]{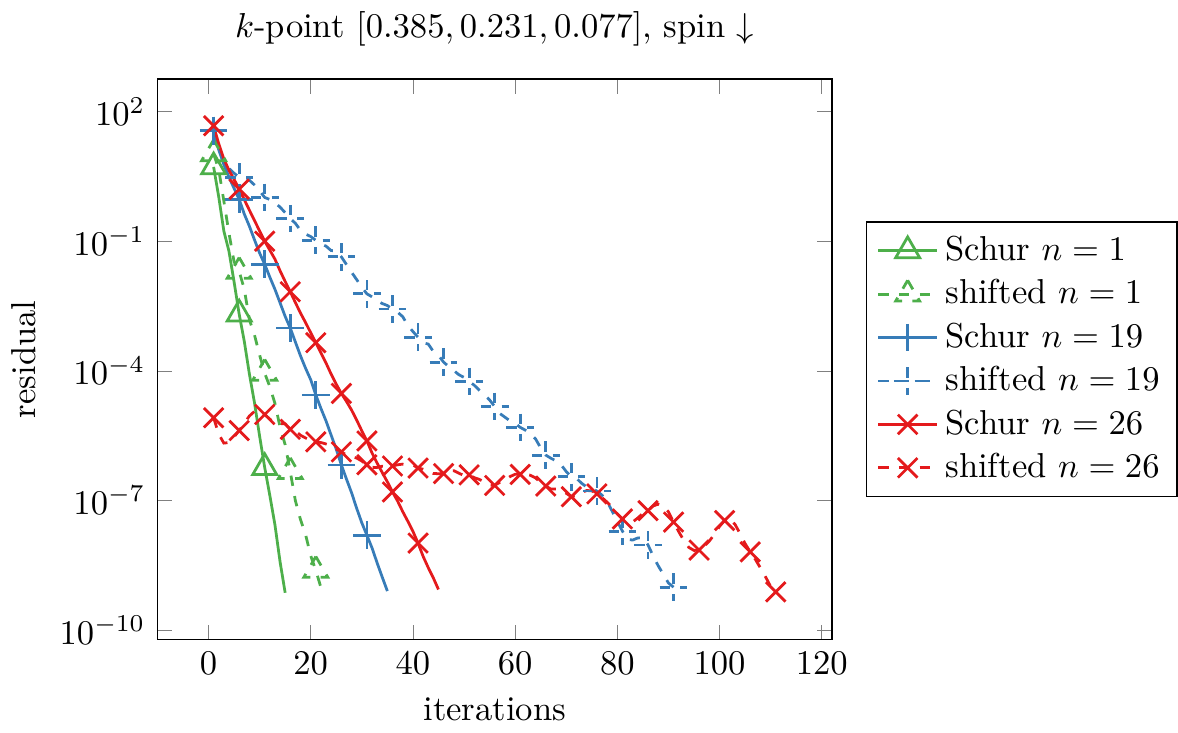}
  \caption{Resolution of the Sternheimer equation for both spin channels of one
    specific $k$-point for the \ch{Fe2MnAl} system, with the Schur approach
    \eqref{eq:Schur} and the shifted approach \eqref{eq:shifted_sternheimer}.
    Note that for this particular $k$-point, the spin $\downarrow$ channel has a
  starting point with already small residual for the highest occupied energy
  level. This is due to the term $f_n$ appearing in
  \eqref{eq:shifted_sternheimer}, and the convergence is still slow. }
  \label{fig:shifted_sternheimer}
\end{figure}
For the orbital responses of the highest-energy occupied bands
the CG iterations on the shifted Sternheimer equation converge very slowly
--- in contrast to the Schur complement approach \eqref{eq:Schur}
we proposed in this work.
In terms of the number of Hamiltonian applications,
the shifted Sternheimer strategy required around $492,000$ applications
versus $208,000$ for the Schur complement approach.

\section{Conclusion}
In density-functional theory the simulation of many physical properties
requires the computation of the response of the ground-state density
to an external perturbation.
In this work we have reviewed the standard formalism of such
response calculations from the point of view of numerical analysis.
We provided an overview of the possible gauge choices for representing
the density response, summarizing and contrasting the approaches employed
by state-of-the-art codes such
as Quantum Espresso~\cite{giannozziQUANTUMESPRESSOModular2009}
or Abinit~\cite{romeroABINITOverviewFocus2020}
in a common framework.
\medskip

Based on our analysis we furthermore suggested two novel approaches
for DFT response calculations.
For the occupied-occupied part of the response we developed
a gauge choice based on the idea to maximize numerical stability
in the involved sums by minimizing the numerical range of the
individual orbital contributions.
For the occupied-unoccupied part of the response we suggested
a novel approach to solving the Sternheimer equation
based on a Schur complement.
Key idea of this approach is to make use of
the additional (partially) converged bands,
which are available as a byproduct from the
preceding self-consistent field (SCF) procedure
(which yields the ground-state density).
Without additional computational effort this allows
to improve the conditioning of the Sternheimer equation and thus
accelerate its convergence.
We demonstrated this numerically on a number of practically relevant problems,
including response calculations on
small-gapped semiconductors, elongated metallic slabs
or numerically challenging Heusler alloy systems.
Overall the Schur complement approach allowed to obtain a converged response
saving up to 40\% in the required Hamiltonian applications
--- the cost-dominating step in small to medium-sized DFT problems.
For larger systems we similarly expect savings from introducing a Schur
complement technique, even though algorithms commonly employ different trade-offs.
\medskip

In this work we followed standard heuristics
for selecting the number of extra bands to employ in the SCF calculations
and thus the number of additional bands available
when solving the response problem.
However, our results emphasize the need for a more robust understanding
between the computed number of bands and the observed rate of convergence.
We have provided some initial ideas for such an analysis in the appendix,
but leave a more exhaustive discussion for future work.

\section*{Acknowledgements}
This project has received funding from the European Research Council (ERC) under the European Union’s Horizon 2020 research and innovation programme (grant agreement No 810367). M.H. and B.S. acknowledge funding by the Federal Ministry of Education and Research (BMBF) and the Ministry of Culture and Science of the German State of North Rhine-Westphalia (MKW) under the Excellence Strategy of the Federal Government and the Länder.

\bibliographystyle{abbrv}
\bibliography{biblio}

\appendix

\section{Choosing the number of extra bands}

In this paper, we saw through various numerical examples that
using a Schur complement to compute the unoccupied-occupied contributions to the
orbitals' response improves the convergence of the Sternheimer equation.
In this appendix, we quantify this acceleration and discuss how this idea
can be used to select the number of bands to be computed.
Considering the straight convergence curves from
Figures~\ref{fig:Al40}--\ref{fig:Fe2MnAl} suggest that the convergence of the
CG is indeed led by the square root of the condition number
of the system matrix (see \cite[Section
9]{shewchukIntroductionConjugateGradient1994}) when using the Schur complement.
Key idea will thus be to estimate the condition number of the linear system
\eqref{eq:Schur}.

\subsection{Numerical analysis}\label{sec:analysis}

To analyse the condition number of the Schur complement, we consider the following
specific setting
\begin{equation}
  \begin{cases}
    H_\rho \phi_n = \varepsilon_n\phi_n,\quad\varepsilon_1 \lq \varepsilon_2 \lq \cdots \\
    \cro{\phi_n,\phi_m} = \delta_{nm},
  \end{cases}
\end{equation}
where $H_\rho\in\C^{\Nb\times\Nb}$
is typically the discretized self-consistent Hamiltonian of the system,
at some $k$-point. We assume that we have $N$ occupied orbitals
that have an occupation number higher than the threshold we fixed and that we
have $\Nex$ extra bands, as explained in \autoref{sec:schur}.
In summary, we have at our disposal $N + \Nex$ bands in total:
$\Phi = (\phi_1,\dots,\phi_N)$ are occupied, fully converged bands and the extra bands
$\Phi_{\rm ex}^{\ell} = (\phi^{\ell}_{N+1},\dots,\phi^{\ell}_{N+\Nex})$ are not necessarily
all converged.  We added here the exponent $\ell$ as we
make the following assumptions:
\begin{itemize}
  \item for any $\ell\in\N$, $(\Phi,\Phi_{\rm ex}^{\ell})$ is an orthonormal family;
  \item for any $\ell\in\N$, $(\Phi_{\rm ex}^{\ell})^* H_\rho\Phi_{\rm ex}^{\ell} \in
    \C^{\Nex\times\Nex}$ is a  diagonal matrix whose elements are labelled
    $\varepsilon_{m}^{\ell} \coloneqq \cro{\phi_m^{\ell},H_\rho\phi_m^{\ell}}$ for
    $N+1\lq m\lq N+\Nex$;
  \item as $\ell\to+\infty$, $(\phi^{\ell}_{m}, \varepsilon^{\ell}_{m}) \to
    (\phi_{m}, \varepsilon_{m})$.
\end{itemize}
All these assumptions are satisfied for instance if the sequence
$(\Phi,\Phi_{\rm ex}^{\ell})_{l\in\N}$ is generated by any Rayleigh-Ritz based eigensolver
(for instance the LOBPCG eigensolver
\cite{knyazevOptimalPreconditionedEigensolver2001}), which is the case by
default in DFTK.
For every $\ell$, we can thus decompose the plane-wave approximation space
$\Hc = X_{\Nb}$ (with $\Nb \gg N+\Nex$) in two different ways:
\begin{equation}
  \Hc = \Ran(P) \oplus \Ran(T) \oplus \Ran(R) \quad\text{and}\quad
  \Hc = \Ran(P) \oplus \Ran(T^{\ell}) \oplus \Ran(R^{\ell}),
\end{equation}
where
\begin{equation}
  P \coloneqq \sum_{n=1}^N \phi_n\phi_n^*\quad\text{and}\quad
  \begin{cases}
    T \coloneqq \sum_{n=N+1}^{\Nex} \phi_n\phi_n^*, \quad R \coloneqq 1 - P - T \\
    T^{\ell} \coloneqq \sum_{n=N+1}^{\Nex} \phi_n^{\ell}(\phi_n^{\ell})^*, \quad R^{\ell} \coloneqq 1 - P - T^{\ell},
  \end{cases}
\end{equation}
are all orthogonal projectors.
In these two
decompositions, $H_\rho$ has the associated block representations:
\begin{equation}
  H_\rho = \begin{pmatrix}
    E & 0 & 0 \\ 0 & E_{\rm ex} & 0 \\ 0 & 0 & \ddots
  \end{pmatrix}
  \quad\text{and}\quad
  H_\rho = \begin{pmatrix}
    E & 0 & 0 \\ 0 & E^{\ell}_{\rm ex} & R^{\ell}H_\rho T^{\ell} \\ 0 & T^{\ell}H_\rho R^{\ell} & R^{\ell}H_\rho R^{\ell}
  \end{pmatrix}
\end{equation}
where $E \coloneqq {\rm Diag}(\varepsilon_1,\dots,\varepsilon_n)$,
$E_{\rm ex} \coloneqq {\rm Diag}(\varepsilon_{N+1},\dots,\varepsilon_{N+\Nex})$ and
$E_{\rm ex}^{\ell} \coloneqq {\rm Diag}(\varepsilon^{\ell}_{N+1},\dots,\varepsilon^{\ell}_{N+\Nex})$
are diagonal matrices. Moreover, note that as
$\Phi^{\ell}_{\rm ex} \to \Phi_{\rm ex}$, the residuals $R^{\ell}H_\rho T^{\ell}$ converge to $0$.

\medskip

Now, we fix $n\lq N$ and we compute the condition number of the linear system
\eqref{eq:Schur}. Enforcing the CG to stay
at each iteration in $\Ran(R^{\ell})$, this condition number is given by the
ratio of the largest and smallest nonzero eigenvalues of
\begin{equation}\label{eq:schur_matrix}
  H_n^{\ell} + X_n^{\ell},
\end{equation}
where
\begin{equation}
  \begin{split}
    H_n^{\ell} \coloneqq R^{\ell}(H_\rho - \varepsilon_n)R^{\ell} \quad\text{and}\quad
    X_n^{\ell} &= -R^{\ell}(H_\rho-\varepsilon_n)\Phi_{\rm ex}^{\ell}
    (E_{\rm ex}^{\ell} - \varepsilon_n)^{-1}(\Phi_{\rm ex}^{\ell})^*(H_\rho - \varepsilon_n)R^{\ell}\\
    &= -R^{\ell}H_\rho\Phi_{\rm ex}^{\ell}(E_{\rm ex}^{\ell} - \varepsilon_n)^{-1}
    (\Phi_{\rm ex}^{\ell})^*H_\rho R^{\ell}.
  \end{split}
\end{equation}
Here $E^{\ell}_{\rm ex}-\varepsilon_n$ is diagonal and thus explicitly invertible if
$\ell$ is large enough as
$\varepsilon_{N+1}^{\ell}\to\varepsilon_{N+1}>\varepsilon_N\gq\varepsilon_n$.
We focus for the moment on the smallest nonzero eigenvalue, that is
$\varepsilon^{\ell}_{N+\Nex+1} - \varepsilon_n$.
The condition number being proportional to the inverse of the smallest eigenvalue,
we now derive a lower bound of $\varepsilon^{\ell}_{N+\Nex+1} - \varepsilon_n$ in
order to get an upper bound on the condition number of \eqref{eq:schur_matrix}.
When $\ell\to+\infty$, we have $X_n^{\ell}\to 0$ (as $RH_\rho\Phi_{\rm ex}=0$) and
$H_n^{\ell}\to H_n\coloneqq R(H_\rho - \varepsilon_n)R$ whose smallest nonzero eigenvalue is
$\varepsilon_{N+\Nex+1} - \varepsilon_n$. We use next a perturbative
approach to effectively approximate the condition number of
\eqref{eq:schur_matrix}.

\medskip

We use a standard eigenvalue perturbation result, whose proof is recalled for the
sake of completeness.  It is directly adapted from the general case of
self-adjoint bounded below operators with symmetric perturbations
studied for instance in \cite{dussonAnalysisFeshbachSchurMethod2020}.
\begin{proposition}\label{prop:perturbation}
  Let $N\in\N$, $H_0,W\in \C^{N\times N}$ be Hermitian matrices and $\alpha\gq0$
  such that $H_0+\alpha>0$. Then, the eigenvalues of $H\coloneqq H_0+W$ and $H_0$ satisfy
  \begin{equation}
    \abs{\nu_i(H) - \nu_i(H_0)} \lq
    (\nu_i(H_0)+\alpha)\norm{W}_{H_0,\alpha},
  \end{equation}
  where $\norm{W}_{H_0,\alpha}$ is the operator norm of
  $(H_0+\alpha)^{-1/2}W(H_0+\alpha)^{-1/2}$ and
  $\nu_i(A)$ is the $i$-th lowest eigenvalue of the matrix $A$.
\end{proposition}
\begin{proof}
  Let $u\in\C^N$ and define $v\coloneqq (H_0+\alpha)^{1/2}u$. Then,
  \begin{equation}
    \begin{split}
      \abs{\cro{u,Hu} - \cro{u,H_0u}} = \abs{\cro{u,Wu}} &=
      \abs{\cro{v, (H_0+\alpha)^{-1/2}W(H_0+\alpha)^{-1/2}v}} \\
      &\lq \norm{W}_{H_0,\alpha}\cro{v,v} =  \norm{W}_{H_0,\alpha}\cro{u,(H_0+\alpha)u}.
    \end{split}
  \end{equation}
  Therefore,
  \begin{equation}
    (1 - \norm{W}_{H_0,\alpha})\cro{u,H_0u} - \alpha\norm{W}_{H_0,\alpha}\cro{u,u}
    \lq \cro{u,Hu} \lq
    (1 + \norm{W}_{H_0,\alpha})\cro{u,H_0u} + \alpha\norm{W}_{H_0,\alpha}\cro{u,u}.
  \end{equation}
  The min-max theorem then yields for $i=1,\dots,N$,
  \begin{equation}
    (1 - \norm{W}_{H_0,\alpha})\nu_i(H_0) - \alpha\norm{W}_{H_0,\alpha}
    \lq \nu_i(H) \lq
    (1 + \norm{W}_{H_0,\alpha})\nu_i(H_0) + \alpha\norm{W}_{H_0,\alpha},
  \end{equation}
  which gives the desired inequality.
\end{proof}

\medskip

In our case, we can apply this result to
\begin{equation}
  H_n^{\ell} + X_n^{\ell} = H_n + (H_n^{\ell} - H_n) + X_n^{\ell},
\end{equation}
with $H_0 = H_n$, $W = W_n^{\ell} \coloneqq (H_n^{\ell} - H_n) + X_n^{\ell}$ and
$\alpha=\varepsilon_{N+\Nex+1} -\varepsilon_n> 0$. \autoref{prop:perturbation}
applied to the
$(N+\Nex+1)$-th eigenvalues then yields
\begin{equation}
  \varepsilon_{N+\Nex+1}^{\ell} -\varepsilon_n\gq
  \prt{\varepsilon_{N+\Nex+1}-\varepsilon_n}
  \prt{1 - 2\norm{W_n^{\ell}}_{H_n,\varepsilon_{N+\Nex+1}-\varepsilon_n}}
  \approx  \prt{\varepsilon_{N+\Nex+1}-\varepsilon_n},
\end{equation}
where we assume that $2\norm{W_n^{\ell}}_{H_n,\varepsilon_{N+\Nex+1}-\varepsilon_n}$
is small enough to be negligible with respect to 1, which is the case if the
extra states are sufficiently converged.
Now, if this bound is valid in theory, in practice we do not have access to
$\varepsilon_{N+\Nex+1}$ as we work with $N+\Nex$ bands
only.   However, up to loosing sharpness, we can use that
$\varepsilon_{N+\Nex+1}\gq\varepsilon_{N+\Nex}$ where
$\varepsilon_{N+\Nex}$ can be estimated using the last extra band.
Indeed, using for instance the Bauer-Fike bound
(\cite[Theorem 1]{herbstPosterioriErrorEstimation2020} or
\cite{saadNumericalMethodsLarge2011}), we obtain
\begin{equation}
  \varepsilon_{N+\Nex} \gq \varepsilon_{N+\Nex}^{\ell}
  - \norm{r_{N+\Nex}^{\ell}},
\end{equation}
where $r_{N+\Nex}^{\ell}$ is the residual associated to the last extra
band. Of course, this estimate is not sharp as we expect the error on the
eigenvalue to behave as the square of the residual, but this requires to
estimate the gap to the rest of the spectrum, see for instance the Kato-Temple
bound \cite[Theorem 2]{herbstPosterioriErrorEstimation2020}.
In the end, we have the following lower bound for
$\varepsilon_{N+\Nex+1}^{\ell}-\varepsilon_n$:
\begin{equation}
  \varepsilon_{N+\Nex+1}^{\ell} -\varepsilon_n\gq \prt{\varepsilon_{N+\Nex}^{\ell}-\varepsilon_n
    - \norm{r_{N+\Nex}^{\ell}}}
  \approx \varepsilon_{N+\Nex}^{\ell}-\varepsilon_n,
\end{equation}
where we assume again that $\norm{r_{N+\Nex}^{\ell}}$ is small enough with respect to
$\varepsilon_{N+\Nex}^{\ell}-\varepsilon_n$.

\medskip

We can now derive an upper bound on $\kappa_{n}^{\ell}$, the condition number of
\eqref{eq:schur_matrix}. It is given by the ratio of its highest eigenvalue
and $\varepsilon_{N+\Nex+1}^{\ell} -\varepsilon_n$. Since the Laplace operator is
the higher-order term in the Kohn-Sham Hamiltonian, the highest
eigenvalue is, as usually in plane-wave simulations, of order $\Ecut$.  With
proper kinetic preconditioning, we can assume that its contribution to the
condition number of the linear system is constant with respect to $\Ecut$ and $n$ so that,
finally,
\begin{equation}
  \kappa_n^{\ell} \lesssim \frac{C}{\varepsilon_{N+\Nex+1}^{\ell} -\varepsilon_n}
  \lesssim \frac{C}{\varepsilon_{N+\Nex}^{\ell}-\varepsilon_n}.
\end{equation}
Therefore the condition number is bounded from above by
$C/(\varepsilon_{N+\Nex}^{\ell}-\varepsilon_n)$ to first-order. The number of CG iterations
to solve the linear system \eqref{eq:Schur} with a given accuracy is then proportional to
the square root of the condition number of the matrix \eqref{eq:schur_matrix}
(see \cite{shewchukIntroductionConjugateGradient1994}):
\begin{equation}\label{eq:kappa}
  \sqrt{\kappa_n^{\ell}} \lesssim
  \sqrt{\frac{C}{{\varepsilon_{N+\Nex}^{\ell}-\varepsilon_n}}}.
\end{equation}
Note that this upper bound is valid provided that the extra bands are converged
enough, not necessarily fully, and proper kinetic preconditioning is employed.

\medskip

Estimate \eqref{eq:kappa} leads, as expected, to the qualitative conclusion
that the more extra bands we use, the higher the difference
$\varepsilon_{N+\Nex}^{\ell}-\varepsilon_n$ and the faster the convergence.
However, note that it is not possible to evaluate directly the convergence
speed as the constant $C$ is \emph{a priori} unknown, in particular if we use
preconditioners.

\subsection{An adaptive strategy to choose the number of extra bands}
\label{sec:adap}

The main bottleneck of \eqref{eq:kappa} is the estimation of the constant $C$.
However, one can reasonably assume that this
constant does not depend too much on $n$, so that the ratio between the number
of iterations to reach convergence between the last occupied band ($n=N$) and the
first band $(n=1)$ can be estimated by
\begin{equation}\label{eq:ratio}
  \xi_{\Nex}^{\ell} \coloneqq \sqrt{\frac{\varepsilon_{N+\Nex}^{\ell}-\varepsilon_1}
    {\varepsilon_{N+\Nex}^{\ell}-\varepsilon_N}},
\end{equation}
This ratio can be of interest as \eqref{eq:kappa} suggests that the Sternheimer
solver converges the fastest for $n=1$ and the slowest for $n=N$.
\medskip

\begin{figure}
  \includegraphics[height=0.33\linewidth]{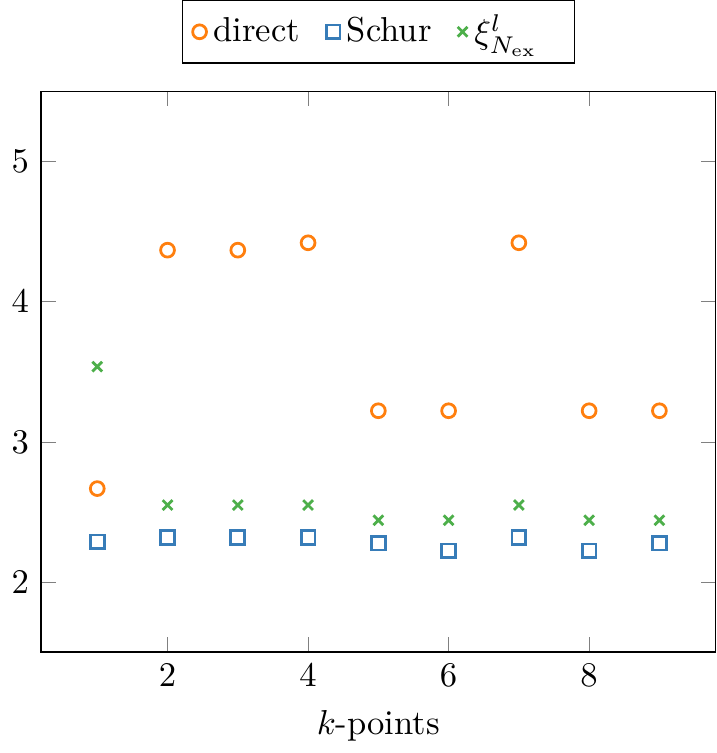}\hfill
  \includegraphics[height=0.33\linewidth]{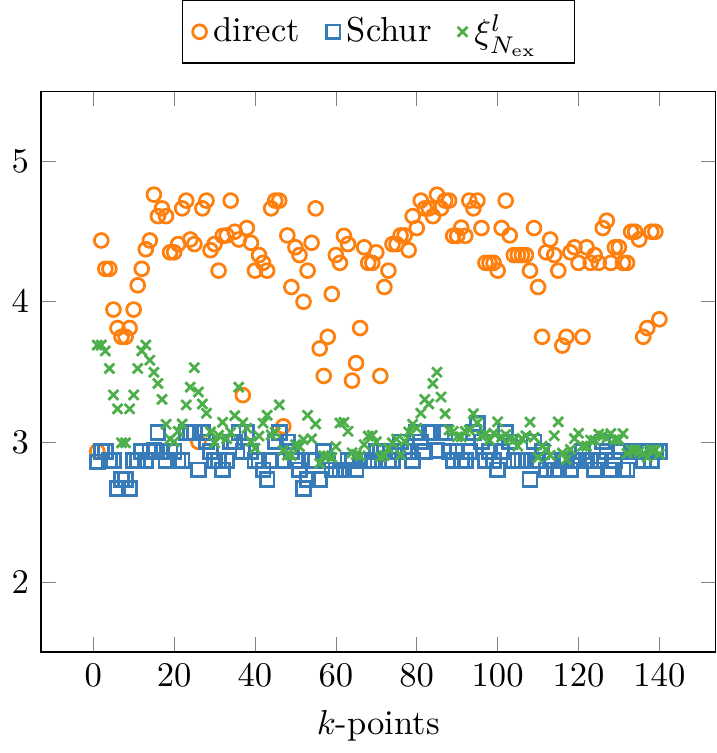}\hfill
  \includegraphics[height=0.33\linewidth]{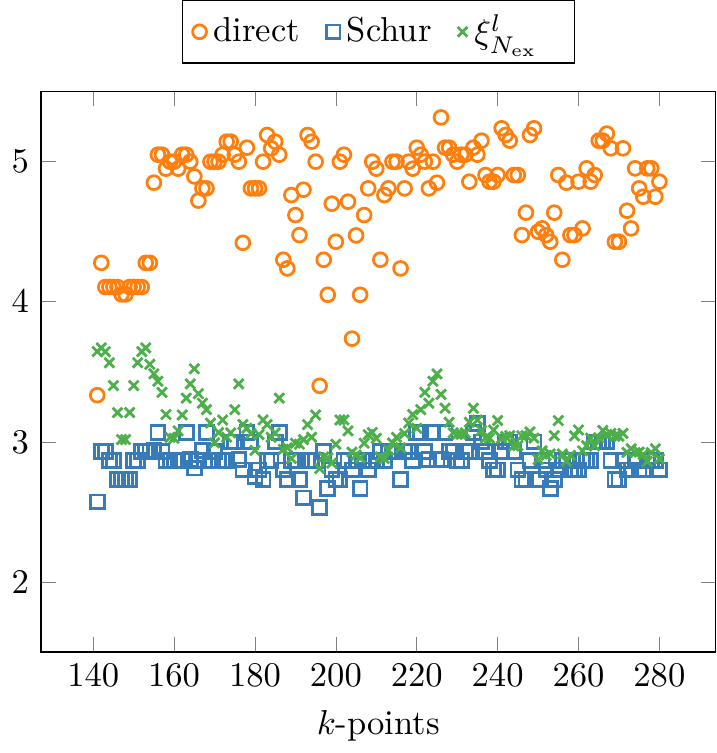}
  \caption{Comparison between the ratio $\xi_{\Nex}^{\ell}$ ($\times$) and the ratios of
    the experimental number of iterations between the first and last occupied bands,
    with ($\square$) and without ($\circ$) the Schur complement
    \eqref{eq:Schur}. On the $x$-axis is the index of the $k$-point. [Left]
    \ch{Al40} [Middle] \ch{Fe2MnAl} spin up channels [Right] \ch{Fe2MnAl}
    spin down channels.}
  \label{fig:ratios}
\end{figure}
We plot in \autoref{fig:ratios} the upper bound $\xi_{\Nex}^{\ell}$ as well as the
computed ratios between the number of iterations of the first and last bands
for the systems we considered in \autoref{sec:num}. These plots show that
$\xi_{\Nex}^{\ell}$ is indeed an upper bound of the actual ratio.  This bound does not
seem to be sharp however.  This is due to the successive approximations we
made to obtain this estimate. Plots in
\autoref{fig:ratios} also confirm that if, for every $k$-point, the ratio of
the number of iterations between the first and last occupied bands is assumed
to be an accurate indicator of the efficiency of the Sternheimer solver, then using the
Schur complement \eqref{eq:Schur} always make this ratio smaller.

\medskip

If we want the ratio of the number of iterations between the first and the
last occupied bands to be lower than some target ratio $\xi_T$ (for instance 3),
\autoref{fig:ratios} suggests that the computable ratio $\xi_{\Nex}^{\ell}$
can help in choosing the number of extra bands to reach this target ratio.
We propose in \autoref{algo:adaptive_strat_nextra} an adaptive algorithm to
select the number of extra bands as a post-processing step after termination of
the SCF.
The basic idea is that, given the initial output
$(\Phi,\Phi_{\rm ex}^{\ell})$ with $\ell=0$ of an SCF calculation,
one iterates $\Phi_{\rm ex}^\ell \to \Phi_{\rm ex}^{\ell+1}$ where $\Phi_{\rm ex}^{\ell}$
gathers the extra bands. At each iteration $\ell$, we compute
$\xi_{\Nex}^{\ell}$ and check if it is below the target ratio.  If not, we compute
more approximated eigenvectors, that we converge until the residual
$\norm{r_{N+\Nex}^{\ell}}$ is negligible with respect to
$\varepsilon_{N+\Nex}^{\ell}-\varepsilon_N$, and so on.
To generate such a residual, after adding a random extra band properly orthonormalized,
we update the extra bands using a LOBPCG with tolerance
\begin{equation}\label{eq:tol}
  {\rm tol} = (\varepsilon_{N+\Nex-1}^{\ell}-\varepsilon_N)/50.
\end{equation}
Note that we use $\varepsilon_{N+\Nex-1}^{\ell}$ instead of $\varepsilon_{N+\Nex}^{\ell}$:
this is done for the sake of simplicity, instead of updating the tolerance on the fly with
$\varepsilon_{N+\Nex}^{\ell}$ changing at each iteration of the LOBPCG.

\medskip
\begin{algorithm}[H]
  \SetAlgoSkip{bigskip}
  \KwData{target ratio $\xi_T$, $\Nex$, $\ell$, $\xi_{\Nex}^{\ell}$}
  \While{$\xi_{\Nex}^{\ell} > \xi_T$ }{
    add random extra band $\phi_{\rm new}$ in the orthogonal of $\Span(\Phi,\Phi_{\rm ex}^{\ell})$;\\
    $\Nex \leftarrow \Nex+1$;\\
    update on the fly the extra bands with tolerance from \eqref{eq:tol}
    using the LOBPCG method; \\
    $\Phi_{\rm ex}^{\ell+1} \leftarrow (\Phi_{\rm ex}^{\ell}, \phi_{\rm new})$ and
    $E_{\rm ex}^{\ell+1} \leftarrow (\Phi_{\rm ex}^{\ell+1})^*H_\rho\Phi_{\rm
      ex}^{\ell+1}$;\\
    $\ell\leftarrow \ell+1$;\\
    compute $\xi_{\Nex}^{\ell}$ with \eqref{eq:ratio};
  }
  \caption{Adaptive choice of the number of extra bands}
  \label{algo:adaptive_strat_nextra}
\end{algorithm}

\subsection{Numerical tests}

We test this strategy on the systems investigated
in \autoref{sec:num}, with different values for the target ratio $\xi_T$ in order
to see a noticeable improvement for each system. In practice, we suggest this
ratio to be between 2 and 3.

\medskip

We first start with the \ch{Al40} system.  \autoref{fig:ratios} [Left] suggests
that the default choice of extra bands
already gives satisfying results by reaching a ratio of approximately 2.5 for
all $k$-points but the $\Gamma$-point (for which there is no real issue with the
Sternheimer equation, according to \autoref{tab:Al40}). We thus run
\autoref{algo:adaptive_strat_nextra} with a smaller target ratio $\xi_T=2.2$.
We use as initial value for $\Nex$ the default value for each $k$-point.
Results
are plotted in \autoref{tab:adapt} [Left] and suggests adding 15 extra
bands. Running again the simulations from \autoref{sec:num} with 72 fully
converged bands and 18 additional, not fully converged, bands yields indeed an
improvement in the convergence of the CG when solving the Sternheimer equation
with the Schur complement method.
Moreover, in \autoref{fig:adapt}, the ratio $\xi_{\Nex}^{\ell}$ indeed
lies below the target ratio $\xi_T=2.2$, and matches this ratio for
the $k$-points that caused difficulties for the Sternheimer equation solver to converge.
In terms of computational time,
the number of Hamiltonian applications to compute the response has been
reduced from $\sim14,800$ with the default number of extra bands to $\sim12,800$. However,
running the algorithm required $\sim3,400$ additional Hamiltonian applications,
making the total amount of Hamiltonian applications higher than that of the Schur
approach with the standard heuristic.

\medskip

\begin{table}
  \small
  \begin{tabular}{ccc}
    \ch{Al40}, $\xi_T = 2.2$ & & \ch{Fe2MnAl}, $\xi_T = 2.5$ \\ \\
    \begin{tabular}{@{}cccc@{}}
\toprule
$k$-point                      & $1$  & $2$  & $5$  \\ \midrule
$N$                            & $69$ & $58$ & $67$ \\
$\text{default } N_{\rm ex}$   & $6$  & $17$ & $8$  \\
$\text{suggested } N_{\rm ex}$ & $21$ & $29$ & $12$ \\
\#iterations $n=1$ Schur    & $21$ & $19$ & $18$ \\
\#iterations $n=N$ Schur    & $32$ & $36$ & $28$ \\\bottomrule
\end{tabular}
 & &
    \begin{tabular}{@{}ccccc@{}}
\toprule
$k$-point / spin                      & $96 \uparrow$ & $96 \downarrow$ & $72 \uparrow$ & $72 \downarrow$ \\ \midrule
$N$                            & $28$          & $26$            & $29$          & $26$            \\
$\text{default } N_{\rm ex}$   & $10$          & $12$            & $9$           & $12$            \\
$\text{suggested } N_{\rm ex}$ & $16$          & $18$            & $17$          & $20$            \\
\#iterations $n=1$ Schur    & $15$          & $15$            & $15$          & $15$            \\
\#iterations $n=N$ Schur    & $36$          & $35$            & $35$          & $35$            \\ \bottomrule
\end{tabular}

  \end{tabular}
  \caption{Suggested number of extra bands for \ch{Al40} and \ch{Fe2MnAl}
    to reach the target
    ratio $\xi_T$, obtained with \autoref{algo:adaptive_strat_nextra} with
    default $\Nex$ as starting point, as well as the number of iterations to
    reach convergence with the newly suggested $\Nex$. Note that the ratio
    between iterations indeed lies below the target ratio $\xi_T$.}
  \label{tab:adapt}
\end{table}
\begin{figure}
  \centering
  \includegraphics[height=0.33\linewidth]{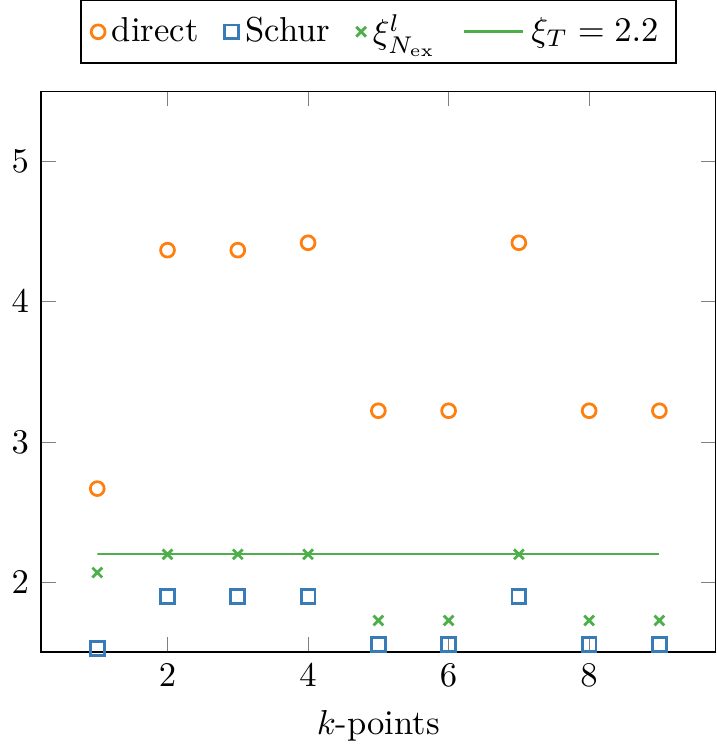}\hfill
  \includegraphics[height=0.33\linewidth]{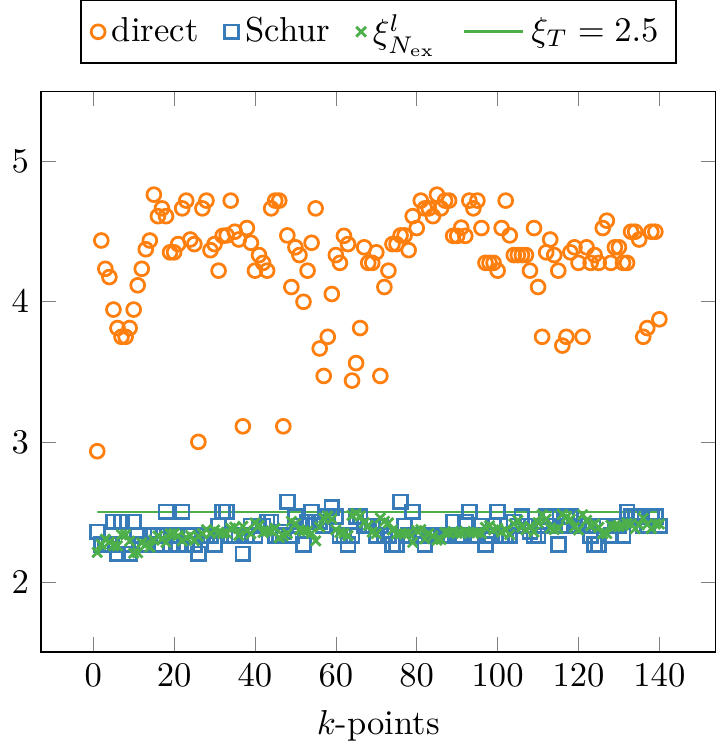}\hfill
  \includegraphics[height=0.33\linewidth]{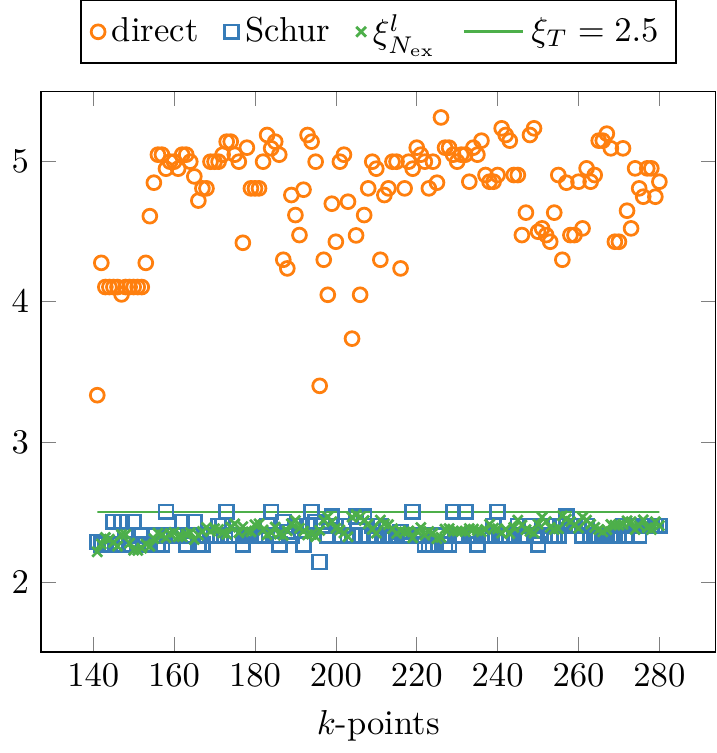}
  \caption{Comparison between the ratio $\xi_{\Nex}^{\ell}$ ($\times$) and the ratios of
    the measured number of iterations between the first and last occupied bands,
    with ($\square$) and without ($\circ$) the Schur complement
    \eqref{eq:Schur}. On the $x$-axis is the index of the $k$-point. [Left]
    \ch{Al40} with 15 additional extra bands. [Left]
    \ch{Fe2MnAl}~spin~$\uparrow$ with 8 additional bands. [Right]
    \ch{Fe2MnAl}~spin~$\downarrow$ with 8 additional bands.}
  \label{fig:adapt}
\end{figure}
Similarly, for \ch{Fe2MnAl}, we run \autoref{algo:adaptive_strat_nextra} with
target ratio $\xi_T=2.5$ as well as initial value the default $\Nex$ for all the 140
$k$-points and spin polarizations.
We present in \autoref{tab:adapt} [Right] the output for both spin
polarizations of two
particular $k$-points. The results are similar for the rest of the $k$-points
and the maximum additional extra bands suggested by the
algorithm is 8. We thus run the same simulations as in \autoref{sec:num} but
this time with 35 fully converged bands and $11$ extra, nonnecessarily
converged, bands. We indeed see for these two $k$-points
that the target ratio has been reached, and that the number of iterations to
converge is smaller than for the default choice we made in
\autoref{tab:Fe2MnAl}.
In \autoref{fig:adapt}, we plot the ratios
$\xi_{\Nex}^{\ell}$ as well as the actual ratios and they almost all lie
below the target ratio.  Contrarily to \ch{Al40}, we note however that the actual
measured ratios are not always below the indicator $\xi_{\Nex}^{\ell}$. In
terms of computational time, the number of Hamiltonian applications
has been reduced from $\sim208,000$ with the
default choice of $\Nex$ to $\sim179,000$. Again, running the algorithm required
$\sim49,000$ additional Hamiltonian applications, making it more expensive
than using the default number of extra bands.

\medskip

It appears that \autoref{algo:adaptive_strat_nextra} can be used to choose the
number of extra bands in order to reach a given ratio $\xi_T$. However, using
the algorithm as such is not useful in practice as it requires a too high
number of Hamiltonian applications, making this strategy less interesting than
the Schur approach we proposed with the default choice of extra bands.
Strategies to reduce the number of Hamiltonian applications in order to choose
an appropriate number of extra bands will be subject of future work.

\end{document}